\documentclass[12pt]{article}

\usepackage{geometry}
\usepackage{chemformula}    
\usepackage[T1]{fontenc}    
\usepackage{amsmath}        
\usepackage{bm}             
\usepackage{bbm}            
\usepackage{amssymb}        
\usepackage{nameref}        
\usepackage{hyperref}       
\usepackage{standalone}
\usepackage{tikz}           
\usepackage{pgfplots}
\usepackage{pgfgantt}       
\usepackage{pifont}         
\usepackage{algorithm}      
\usepackage{algpseudocode}
\usepackage{nomencl}        
\usepackage{etoolbox}
\usepackage{amsthm}
\usepackage{tabularx}       
\usepackage{float}          
\usepackage{longtable}
\usepackage[numbers,super,comma,sort&compress]{natbib}
\usepackage{authblk}
\usepackage{doi}

\makenomenclature
\newtheorem{theorem}{Theorem}
\theoremstyle{definition}
\newtheorem{definition}{Definition}

\DeclareMathOperator*{\argmax}{arg\,max}
\pgfmathdeclarefunction{gauss}{2}{%
    \pgfmathparse{1/(#2*sqrt(2*pi))*exp(-((x-#1)^2)/(2*#2^2))}%
}
\pgfmathsetseed{2}

\renewcommand\nomgroup[1]{%
    \item[\bfseries
    \ifstrequal{#1}{D}{Continuous Variables}{%
    \ifstrequal{#1}{B}{Sets}{%
    \ifstrequal{#1}{A}{Indices}{%
    \ifstrequal{#1}{C}{Discrete Variables}{%
    \ifstrequal{#1}{E}{Parameters}{}}}}}%
]}

\begin{document}
\title{Data-driven optimization of processes with degrading equipment}
\author[1]{Johannes Wiebe}
\author[2]{In\^es Cec\'ilio}
\author[1]{Ruth Misener}
\affil[1]{Department of Computing, Imperial College London, London, UK}
\affil[2]{Schlumberger Cambridge Research, Cambridge, UK}
\date{}

\maketitle

\begin{abstract}
    In chemical and manufacturing processes, unit failures due to equipment degradation can lead to process downtime and significant costs.
    In this context, finding an optimal maintenance strategy to ensure good unit health while avoiding excessive expensive maintenance activities is highly relevant.
    We propose a practical approach for the integrated optimization of
    production and maintenance capable of incorporating uncertain sensor data regarding equipment degradation.
    To this end, we integrate data-driven stochastic degradation models from
    Condition-based Maintenance into a process level mixed-integer optimization problem using Robust Optimization.
    We reduce computational expense by utilizing both analytical and data-based approximations and optimize the Robust Optimization parameters using Bayesian Optimization.
    We apply our framework to five instances of the State-Task-Network and demonstrate that it can efficiently compromise between equipment availability and cost of maintenance.
\end{abstract}

\section{Introduction}

Most technical processes contain equipment which degrades over time due to its
usage. Degradation may lead to serious equipment failures, unless
\textbf{preventive maintenance} actions are scheduled regularly to restore
equipment conditions. While frequent preventive maintenance can keep equipment
availability high, it also incurs significant cost. At the same time,
unexpected equipment failures can lead to loss of production and high
\textbf{corrective maintenance} costs. Finding the optimal balance between
preventive and corrective maintenance is difficult, because degradation tends
to be at least partially random and the health state of equipment can often
only be estimated from data subject to uncertainty. To make things worse, scheduling maintenance
activities is not independent from production planning and scheduling. A unit
undergoing maintenance might, for example, be unavailable for production.
Furthermore, the state of equipment health tends to depend not only on the
selected maintenance strategy, but also on the process operating strategy.
Operating a process with a high throughput might enable higher production
volumes and more sales, but could also cause more equipment degradation and
therefore a higher maintenance cost. These interactions between process
conditions, maintenance strategy, and the equipment's uncertain
state of health make finding optimal and compatible maintenance
and operating strategies a very challenging data-driven optimization problem
under uncertainty.

One way of reducing the equipment maintenance cost is to determine
maintenance schedules based on information regarding the equipment's state of
health collected through condition monitoring \cite{Jardine2006}. This is the
\textbf{Condition-based maintenance} (CBM) paradigm \cite{Jardine2006,
Barraza2014, Bousdekis2015, Alaswad2016}. Due to the increased availability of
cheap sensors and thereby large quantities of system health data, CBM is
becoming more attractive \cite{Meeker2014}. While much attention has been paid
to data collection \& processing and prognostic modeling, the objective is
usually to minimize the cost of maintaining a single unit \cite{Alaswad2016}.
This means that interaction between maintenance strategies for each piece of
equipment and the operating strategy of the entire process has largely been
neglected.

However, these interactions have been considered by multiple authors in the context
of \textbf{integrated maintenance scheduling and process optimization}.
Early approaches in this field which explicitly model degradation
assume constant, known reliability or decay
curves\citep{Dedopoulos1995,Dedopoulos1995b,Vassiliadis1999,Vassiliadis2001,Casas2005}.
\citet{Dedopoulos1995,Dedopoulos1995b} combine short-term stochastic
scheduling with long-term maintenance scheduling in a two-step procedure, while
\citet{Vassiliadis2001} determine optimal availability thresholds at which maintenance should be performed.
\citet{Georgiadis2000} optimize
the cleaning and energy management of heat exchanger networks subject to
fouling, which is assumed to follow a known profile. \citet{Liu2014} consider
scheduling of maintenance and biopharmaceutical batch production with a
deterministic performance decay. \citet{Xenos2016} optimize maintenance and
production scheduling of a compressor network. The power consumed by the
compressors is assumed to increase linearly with operating time (since
maintenance was performed) due to fouling. \citet{Zulkafli2016,Zulkafli2017}
develop an optimization framework for simultaneous operational planning and
maintenance scheduling of production and utility systems. They consider extra
energy costs caused by performance degradation. The degradation is assumed to
depend on operating time and the production rate. \citet{Aguirre2018} consider
integrated planning, scheduling and maintenance under schedule-dependent,
deterministic performance decay. \citet{Rajagopalan2017} analyze turnaround
rescheduling and apply stochastic programming to manage unplanned
outages.
\citet{Biondi2017} extend the
State Task Network (STN), originally proposed by \citet{Kondili1993}, to
account for degrading equipment and different operating modes. They assume that
each unit, after maintenance, has a given maximum residual lifetime and that
each task performed on a unit in a certain operating mode reduces this residual
lifetime by a given amount. Noticeably, none of these authors make use of the
wealth of knowledge regarding degradation modeling and inference from data
available from the CBM literature. Furthermore, degradation is assumed to be
deterministic which may not be the case in practice.

Recent works have started to incorporate degradation models from CBM into
process level Mixed-Integer Linear Programming (MILP)
problems\citep{Yildirim2016_1,Yildirim2016_2,Yildirim2017,Basciftci2018,Verheyleweghen2017}.
\citet{Yildirim2016_1, Yildirim2016_2} formulate an optimization model for
generator maintenance and production scheduling. The cost of maintenance is
calculated beforehand using a data-driven degradation model:
\begin{equation*}
    c_t = \frac{c^{prev} \left(1 - p_t^{f}\right) + c^{corr}
        p_t^{f}}{\int_0^t{p_{\tau}^{f}d\tau}},
\end{equation*}
where $c_t$ is the predicted cost of performing maintenance at
time $t$ given the cost of preventive $\left(c^{prev}\right)$ and corrective
$\left(c^{corr}\right)$ maintenance. The failure probability $p^{f}_t =
P(\text{unit fails before }t)$ is calculated from the degradation model. The
authors later applied the same approach to maintenance and operation of wind
farms and extended it to include opportunistic maintenance \cite{Yildirim2017}.
While this approach starts to incorporate information from more sophisticated
degradation models into process level optimization problems, degradation is
still considered to be deterministic in the MILP optimization.
\citet{Basciftci2018} extend this to consider sudden failures by using
stochastic programming and generating scenarios from the underlying degradation
model. To the best of our knowledge Ba{\c{s}}{\c{c}}iftci et al.'s{\cite{Basciftci2018}} approach is the
only work combining stochastic
optimization with information from degradation models.

Unfortunately, the aforementioned approach cannot capture effects
of the selected operating strategy on degradation. Since maintenance cost is
calculated based on the degradation model before the optimization problem is
solved, the degradation is assumed to be independent of the operating strategy.
In practice this will often not be the case.

This paper argues for a tighter integration between the sophisticated
degradation models used in CBM and process level maintenance scheduling and
process optimization. To this end we make multiple contributions:
\begin{itemize}
    \item We show how L\'evy type models, a class of stochastic processes
        commonly used in Degradation Modeling, can be incorporated into an integrated maintenance and process MILP model.
        L\'evy type models include the Wiener and Gamma processes -- two very popular models in CBM\@.
        By making the L\'evy models parameters depend on a set of operating modes, equipment degradation too depends on the operating strategy.
    \item We show how uncertainty and randomness in the equipment's degradation characteristics can be incorporated using adjustable robust optimization.
        We use results from the CBM literature to efficiently determine the
        robustness of the obtained solution.
    \item We prove that, in certain cases, feasible solutions to the adjustable robust optimization problem can be found by solving a deterministic approximation with worst case values for the uncertain parameters.
    \item Realizing that process planning and scheduling can be computationally expensive yet highly repetitive, we develop a computationally efficient, data-driven way of a-priori estimating equipment failure probabilities.
        To this end, we generate data using a short-term scheduling model repeatedly.
        Using this data, we propose two methods based on Logistic Regression
        capable of cheaply generating a large number of long-term schedules which can be used to estimate failure probabilities.
    \item We propose Bayesian optimization for efficiently optimizing the uncertainty set.
        The uncertainty set size depends on a small number of parameters, but solving the robust MILP integrated maintenance and process optimization problem can be computationally expensive.
        Bayesian optimization is ideal for this kind of low dimensional problem with expensive function evaluations.
\end{itemize}

As a challenging case study, we apply the proposed method to an extension of the state-task-network (STN) \cite{Biondi2017, Kondili1993}.
This model combines both planning and scheduling of production and maintenance with operating mode dependent equipment degradation.
We test our method on a number of STN instances \cite{Lappas2016, Kondili1993, Karimi1997, Maravelias2003, Ierapetritou1998}.

\section{Combining degradation modeling and robust optimization}

Following \citet{Vassiliadis2001}, we assume an integrated production and maintenance scheduling problem of the form
\begin{align}
    \label{eq:general}
    & \underset{\bm{x},\bm{m}}{\text{min}}
    &   & \text{cost}(\bm{x}, \bm{m})                        \\
    & \text{s.t.}
    &   & \text{process model}(\bm{x}, \bm{m}) \tag{1a}      \\
    &
    &   & \text{maintenance model}(\bm{x}, \bm{m}), \tag{1b}
\end{align}
where $\bm{x}$ are the process variables (continuous and discrete) and $\bm{m}$ are the maintenance related variables.
The process model includes, e.g., material balances, energy balances, unit constraints, and the maintenance model includes, e.g., maintenance crew constraints or constraints regarding different types of maintenance.
Note that cost minimization could easily be replaced by profit maximization.

A health model added to Problem~\ref{eq:general} accounts for equipment degradation:
\begin{align}
    \label{eq:genhealth}
    & \underset{\bm{x},\bm{m}, \bm{h}}{\text{min}}
    &   & \text{cost}(\bm{x}, \bm{m}, \bm{h})                                      \\
    & \text{s.t.}
    &   & \text{process model}(\bm{x}, \bm{m}, \bm{h})\tag{2a}                     \\
    &
    &   & \text{maintenance model}(\bm{x}, \bm{m}, \bm{h})\tag{2b}                 \\
    &
    &   & \text{health model}(\bm{x}, \bm{m}, \bm{h})\tag{2c}\label{eq:healthmod},
\end{align}
where $\bm{h}$ are health related variables and the health model includes all equipment health or degradation related constraints.
Our first contribution is developing a generic health model based on the assumption that the equipments' state of health can be described by L\'evy type processes, a class of stochastic processes commonly used for modeling degradation in CBM.

\subsection{Degradation Modeling}
The premise in Degradation Modeling is that a degradation signal $s^{meas}(t)$ describes the state of degradation of a unit over time.
Signal $s^{meas}(t)$ can either be measured directly or obtained indirectly from measurements.
Two common assumptions adopted in this paper are that:

\begin{description}
    \item[(SMAX) \label{eq:smax}] The unit fails and requires corrective
        maintenance when $s^{meas}(t)$ crosses a
        threshold $s^{max}$ \cite{Wang2007},
    \item[(AGAN) \label{eq:agan}] $s^{meas}(t)$ is reset back to initial value
        $s^{0}$ after preventive maintenance.
        The unit is as-good-as-new (AGAN) \cite{Doyen2004}.
\end{description}

The degradation signal $s^{meas}(t)$ is often modeled by stochastic processes\cite{Alaswad2016}.
One class of stochastic processes are L\'evy type processes:
\theoremstyle{definition}
\begin{definition}{\textbf{L\'evy type process} \cite{Applebaum2004}.}
    A stochastic process $S(t) = \{S_{t}: t \in T\}$, where $S_{t}$ is a random variable, with
    \begin{enumerate}
        \item independent increments: $S_{t_2} - S_{t_1}, \ldots, S_{t_n} - S_{t_{n-1}}$ are independent for any $0 < t_1 < t_2 < \ldots < t_n < \infty$,
        \item stationary increments: $S_t - S_s$ and $S_{t-s} - S_0$ have the same distribution for any $s<t$,
        \item continuity in probability: $\lim_{h \rightarrow 0} P(|{S_{t+h} - S_t}| > \epsilon) = 0$ for any $\epsilon > 0$, $t \geq 0$.
    \end{enumerate}
\end{definition}

L\'evy type processes include both the Wiener and Gamma processes, which are the most commonly used stochastic processes in the Degradation Modeling literature \cite{Alaswad2016, Ye2015, Si2011, Nguyen2018}.
Due to their independence and stationarity, L\'evy type process increments can be described by
\begin{equation}
    \label{eq:levy}
    \begin{aligned}
          & S_{t} - S_{t-\Delta t} = D(\Delta t), \quad &   & D(\Delta t) \sim
        \mathcal{D}(\boldsymbol{\theta}, \Delta t), \quad \forall t,
    \end{aligned}
\end{equation}
where $D(\Delta t)$ is a random variable that follows a given distribution $\mathcal{D}(\boldsymbol{\theta}, \Delta t)$ with parameters $\boldsymbol{\theta}$.
A difficulty, however, arises when $\mathcal{D}$ is also
dependent on some of the operational variables $\boldsymbol{x}$:
\begin{equation}
    \label{eq:levyenv}
    D(\Delta t) \sim \mathcal{D}(\boldsymbol{\theta}(\boldsymbol{x}), \Delta t).
\end{equation}
This dependence has been addressed by assuming that the operational variables
$\boldsymbol{x}$ are piecewise constant, i.e., the process can only operate in a
number of discrete operating modes $k \in K$\cite{Liao2013, Li2016}. Under this
assumption Eqns.~\ref{eq:levy} and~\ref{eq:levyenv} simplify to
\begin{equation}
    \label{eq:levyom}
    \begin{aligned}
        & S_t - S_{t-\Delta t} = \sum_{k \in \mathcal{K}} x_{k,t} \cdot D_k(\Delta t),
        && D_k(\Delta t) \sim \mathcal{D}(\boldsymbol{\theta}_k, \Delta t),
    \end{aligned}
\end{equation}
where $x_{k,t}$ is $1$ if the process operates in mode $k$ at time $t$ and $0$
otherwise. Note that this approach is very similar to regime-switching L\'evy
models used extensively in finance \cite{Chevallier2017}. \citet{Biondi2017} use a
similar approach in their STN extension.

Much of the Degradation Modeling literature focuses on estimating $\boldsymbol{\theta}$ and using, e.g., Bayesian approaches to update it regularly based on new available data \cite{Gebraeel2005, Bian2012}.
A major advantage of Bayesian approaches is that $\boldsymbol{\theta}$ can be estimated based on a population of units first and then individually adjusted to a particular unit \cite{Gebraeel2008}.

\subsection{Constructing a health model}
We summarize the assumptions on which health model~\ref{eq:healthmod} hereafter
is based: For each process unit $j$, a degradation signal $s^{meas}_{j}(t)$ can be
obtained from measurements which is modeled well by a L\'evy process
$S_{j}(t)$, i.e., increments follow Eqn.~\ref{eq:levyom}. The unit fails when
$S_{j}(t)$ reaches a maximum threshold $s_{j}^{max}$~\nameref{eq:smax}
($T^{fail} = \inf\{t \in T|S_{j,t} > s_j^{max}\}$) and $S_{j}(t)$ resets to an
initial value $s_{j}^{0}$ after maintenance~\nameref{eq:agan}. Based on these
assumptions and assuming a discrete time formulation, the following health
model replaces Eqn.~\ref{eq:healthmod}:
\begin{equation}
    \label{eq:healthmodlevy}
    \begin{aligned}
        &
        && S_{j,t} \leq s_{j}^{max}
        &   &   & \forall t, j \in J  \\
        &
        && S_{j,t} =
        \begin{cases}
            S_{j,t-1} + \sum_{k \in \mathcal{K}}{x_{j,k,t}\cdot D_{j,k}}, & \text{if }
            m_{j,t} = 0\\ s_{j}^{0}, & \text{otherwise}
        \end{cases}
        &   &   & \forall t, j \in J, \\
    \end{aligned}
\end{equation}
where $J$ is the set of process units and $m_{j,t}$ is $1$ if a maintenance action starts on unit $j$ at time $t$ and $0$ otherwise.
To address the random nature of degradation, the random variables $D_{j,k}$ and
$S_{j,t}$ can be approximated by an uncertain parameter $\tilde{d}_{j,k}$ and a
deterministic variable $s_{j,t}$ respectively.
Assuming that $\tilde{d}_{j,k}$ is bounded by a compact uncertainty set $\mathcal{U}$, Problem~\ref{eq:genhealth} can be robustified by requiring that all constraints hold for any $\tilde{d}_{j,k}
\in \mathcal{U}$:
\begin{equation}
    \label{eq:combinedu}
    \begin{aligned}
        &
        && s_{j,t} \leq s_{j}^{max}
        &   &   & \forall t, j \in J                                   \\
        &
        && s_{j,t} =
        \begin{cases}
            s_{j,t-1} + \sum_{k \in \mathcal{K}}{x_{j,k,t}\cdot \tilde{d}_{j,k}},
            & \text{if } m_{j,t} = 0\\ s_{j}^{0}
            & \text{otherwise}
        \end{cases}
        &   &   & \forall \tilde{d}_{j,k} \in \mathcal{U}, t, j \in J. \\
    \end{aligned}
\end{equation}
This model explicitly considers preventive maintenance.
    Corrective maintenance becomes necessary only when realizations of
    $\tilde{d}_{j,k}$ lie outside the uncertainty set $\mathcal{U}$ and
    constraint $s_{j,t} \leq s^{max}_{j}$ is violated.

Notice that it is generally not possible to choose $s_{j,t}$ such that the
equality constraint in Problem~\ref{eq:combinedu} holds for all values of
$\bar{d}_{j,k}$ in $\mathcal{U}$, except for the trivial solution $x_{j,k,t} =
0, \forall j,k,t$.
This is because the degradation signal $s_{j,t}$ is an analytical variable, not a decision variable.
Interpreting the degradation signal instead as a second stage variable
$s_{j,t}\left(\tilde{d}_{j,k}\right)$ turns Problem~\ref{eq:combinedu} into an
adjustable robust optimization problem and a linear decision rule can be used
to expresses $s_{j,t}$ as a function of $\tilde{d}_{j,k}$\cite{Lappas2016}:
\begin{equation}
    \label{eq:lindec}
    s_{j,t}\left(\tilde{d}_{j,k}\right) = [s_{j,t}]_{0} + \sum_{k}{[s_{j,t}]_{k}\tilde{d}_{j,k}},
\end{equation}
where $[s_{j,t}]_{0}$ and $[s_{j,t}]_k$ are coefficients which become variables
in the adjustable robust problem.
Technically, $\tilde{d}_{j,k}$ should also be indexed by $t$ as every time
period constitutes an independent realization of $D_{j,k}$. Time-indexed
uncertain parameters have been previously explored\citep{Lappas2016}, but they can lead to a
large increase in variables, especially for discrete time formulations. We therefore make the simplifying assumption that
uncertainty is only revealed once after all variables except $s_{j,t}$ have
been selected.

The health model~\ref{eq:combinedu} can be reformulated to remove the conditonal equality constraint, resulting in the final formulation:
\begin{equation}
    \label{eq:rob}
    \begin{aligned}
        &\min_{\bm{x}, \bm{m}}
        && \text{cost}(\bm{x}=[x_{j,k,t},\ldots]^{\top},
        \bm{m}=[m_{j,t},\ldots]^{\top},
        \bm{h}=[s_{j,t}\left(\bar{d}_{j,k}\right)]^{\top})
        &   &   &                                        \\
        & \text{s.t}
        && \text{process model}(\bm{x}, \bm{m}, \bm{h})
        &   &   &                                        \\
        &
        && \text{maintenance model}(\bm{x}, \bm{m}, \bm{h})
        &   &   &                                        \\
        &
        && m_{j,t} s_{j}^{0} \leq s_{j,t}\left(\bar{d}_{j,k}\right) \leq s_{j}^{max} + m_{j,t} \cdot (s_{j}^{0} -
        s_{j,max})
        &   &   & \forall t, j \in J, \tilde{d} \in \mathcal{U}\\
        &
        && s_{j,t}\left(\bar{d}_{j,k}\right) \geq s_{j,t-\Delta t} + \sum_{k}{x_{j,k,t}\tilde{d}_{j,k}} + m_{j,t}
        \cdot (s_{j}^{0} - s_{j}^{max})
        &   &   & \forall t, j \in J, \tilde{d} \in
        \mathcal{U}\\
        &
        && s_{j,t}\left(\bar{d}_{j,k}\right) \leq s_{j,t-\Delta t} + \sum_{k}{x_{j,k,t}\tilde{d}_{j,k}}
        &   &   &\forall t, j \in J, \tilde{d} \in \mathcal{U}.\\
    \end{aligned}
\end{equation}
By replacing $s_{j,t}\left(\tilde{d}_{j,k}\right)$ with Eqn.~\ref{eq:lindec} in
each constraint and using standard robust optimization reformulation
techniques, the health model can be transformed into a deterministic robust
counterpart (see Appendix~\ref{sec:appa}).

Consider a deterministic version of Problem~\ref{eq:rob} in which cost, process
model, and maintenance model are not functions of $s_{j,t}\left(\tilde{d}_{j,k}\right)$ and
$\tilde{d}_{j,k}$ has been replaced by $d^{max}_{j,k} = \max_{\mathcal{U}} \tilde{d}_{j,k}$:
\begin{equation}
    \label{eq:det}
    \begin{aligned}
        &\min_{\bm{x}, \bm{m}}
        && \text{cost}(\bm{x}, \bm{m})
        &   &   &           \\
        & \text{s.t}
        && \text{process model}(\bm{x}, \bm{m})
        &   &   &           \\
        &
        && \text{maintenance model}(\bm{x}, \bm{m})
        &   &   &           \\
        &
        && m_{j,t} s_j^0 \leq s_{j,t},
        &   &   & \forall t, j \in J\\
        &
        && s_{j,t} \leq s_j^{max} + m_{j,t}(s_j^0 - s_j^{max}),
        &   &   & \forall t, j \in J\\
        &
        && s_{j,t} \geq s_{j,t-1} + \sum_{k}x_{j,k,t}d_{j,k}^{max} + m_t{j,t}(s_j^0-s_j^{max}),
        &   &   & \forall t, j \in J\\
        &
        && s_{j,t} \leq s_{j,t-1} + \sum_{k}x_{j,k,t}d_{j,k}^{max},
        &   &   & \forall t, j \in J,
    \end{aligned}
\end{equation}
where $s_{j,t}$ is not a second stage variable anymore since there are no more
semi-infinite constraints.
Under certain circumstances, feasible solutions to robust Problem~\ref{eq:rob} can be found by solving deterministic Problem~\ref{eq:det}:
\begin{theorem}
    \label{theo:rob-det}
    Given that cost, process model, and maintenance model are not functions of
    $\tilde{d}_{j,k}$ and that $s^0_j \leq s^{init}_j = s_{j,t=t_0} \leq s^{max}_j$ and $\tilde{d}_{j,k} \geq 0, \forall \tilde{d}_{j,k} \in \mathcal{U}$, then a feasible solution $(\bm{x} = [x_{k,t},\ldots], \bm{m}=[m_t,\ldots], \bm{h} = [s_t])$ to Problem~\ref{eq:det} forms a feasible solution $(\bm{x} = [x_{k,t},\ldots], \bm{m}=[m_t,\ldots], \bm{h} = [[s_t]_0, [s_t]_k])$ to Problem~\ref{eq:rob} with
\begin{subequations}
    \begin{align}
        [s_t]_0 &=
        \begin{cases}
            s^{init} & t < t_{m,0}    \\
            s^{0}    & t \geq t_{m,0}
        \end{cases}\\
        [s_{t}]_k &= \sum_{t'=t_{m,t}}^t{x_{k,t}},
    \end{align}
\end{subequations}
where $s^{init} = s(t=0)$, $t_{m,0}$ is the first point in time at which maintenance is performed, and $t_{m,t}$ is the most recent point in time at which maintenance was performed.
\end{theorem}
\begin{proof}
    See Appendix~\ref{sec:appb}
\end{proof}
Theorem~\ref{theo:rob-det} only guarantees solution feasibility, not
optimality. How well Problem~\ref{eq:det} approximates Problem~\ref{eq:rob}
also depends on the selected uncertainty set.

\subsection{The uncertainty set}
A major decision in robust optimization is the uncertainty set choice.
This paper uses a simple box uncertainty set
\begin{equation*}
    \mathcal{U} = \{\tilde{d}_{j,k}|\bar{d}_{j,k}(1-\epsilon_{j,k}) \leq
    \tilde{d}_{j,k} \leq \bar{d}_{j,k}(1+\epsilon_{j,k})\},
\end{equation*}
where $\bar{d}_{j,k}$ is the nominal value of $\tilde{d}_{j,k}$ and $\epsilon_{j,k}$ is a parameter determining the uncertainty set size.
Note that this choice assumes that the random increments $D_{j,k}$
are independent, as a box uncertainty set cannot capture
correlation between uncertain parameters. This assumption could be relaxed with
a more complicated uncertainty set, e.g., a polyhedral set.
Since Degradation Modeling assumes that the distribution of $D_{j,k}$ is known, $\epsilon_{j,k}$ can be determined using the inverse cumulative distribution function~$F^{-1}$:
\begin{equation*}
    \epsilon_{j,k} = 1 - F^{-1}(\alpha)/\bar{d}_{j,k},
\end{equation*}
where $\alpha = P(D_{j,k} \leq \bar{d}_{j,k}(1-\epsilon_{j,k}))$.
If the distribution of $D_{j,k}$ is unknown, data-driven non-parametric methods such as Kernel Density Estimation can be used to estimate it \cite{Ning2017}.

By using $F^{-1}$, the uncertainty set size depends on a single parameter $\alpha \in [0, 0.5]$.
For $\alpha=0$, the uncertainty set includes all possible realizations of $D_{j,k}$ and for $\alpha=0.5$ the uncertainty set is a singleton and the robust optimization problem is equivalent to the deterministic problem using the nominal values $\bar{d}_{j,k}$.
While box uncertainty sets are often more conservative than most of the many
other available uncertainty set types\cite{Guzman2015,Li2011}, the solution
robustness/conservatism in this formulation can be varied by adjusting
$\alpha$.

\subsection{Evaluating robustness}
Assume $\boldsymbol{x}^k_j = [k_1, k_2, \ldots, k_{T}]$, where $k_t = k \iff x_{j,k,t} = 1$, is the sequence of operating modes given by a solution to Problem~\ref{eq:rob}.
Its robustness can be measured by the probability $p_j$ that unit $j$ does not fail in the time horizon $T$
\begin{equation*}
    p_j = P(S_{j,t} \leq S_{j,max}, \forall t < T | \boldsymbol{x}^k_j),
\end{equation*}
or equivalently its probability of failure
\begin{equation*}
    p^f_j = 1-p_j = P(\exists t  \text{ such that } S_{j,t} > S_{j,max} | \boldsymbol{x}^k_j).
\end{equation*}
Assuming the parameters $\boldsymbol{\theta}_{j,k}$ of the distributions $\mathcal{D}_{j,k}$ are estimated from data, $p^f_j$ can be calculated through Monte-Carlo simulation by randomly generating $N$ realizations $\boldsymbol{s}^n_j = [ s^n_{j,0}, s^n_{j,\Delta t}, \ldots, s^n_{j,T}]$ of $S_j(t,\boldsymbol{x}^k_j)$ with $\mathcal{D}_{j,k}(\boldsymbol{\theta}_{j,k}, \Delta t)$ distributed increments and checking how many violate $s^n_{j,t} \leq s^{max}_j$:
\begin{equation}
    \label{eq:pfail}
    p^f_{j} = \frac{\sum_{n=1}^N\mathbbm{1}(\exists t<T \text{ such that }
        s^n_{j,t} > s^{max}_j)}{N},
\end{equation}
where $\mathbbm{1}$ is the indicator function. This is illustrated in Fig.~\ref{fig:calcp} for $N=3$.
\begin{figure}[p]
    \centering
    \includegraphics[scale=0.98]{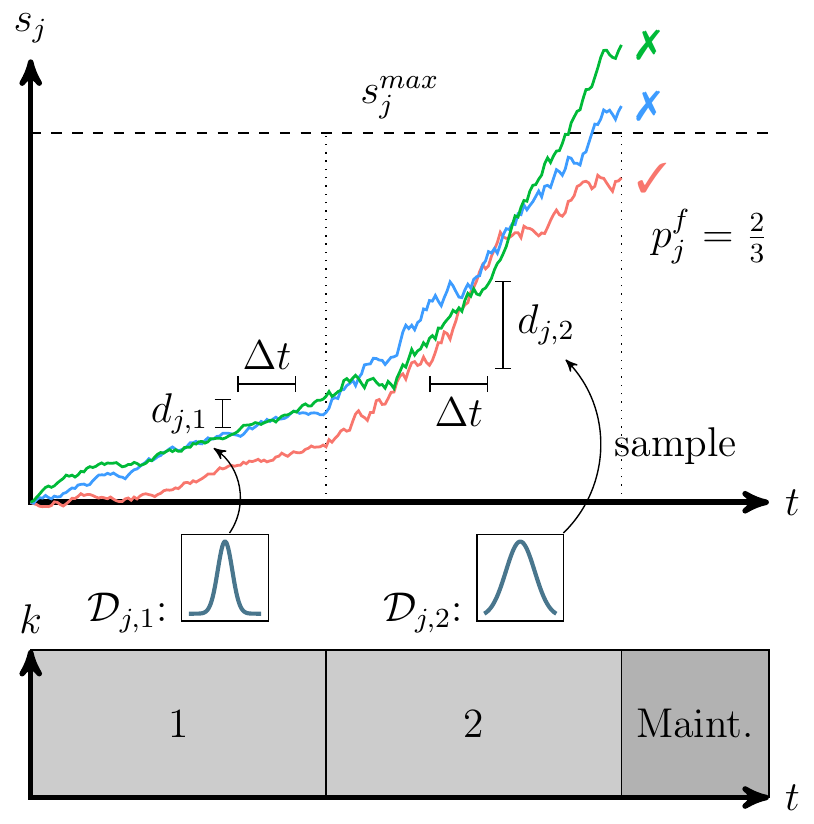}
    \caption{Example for calculating the failure probability $p^f_j$ using
        Eqn.~\ref{eq:pfail} and Monte-Carlo simulation. The operating mode
    schedule is $\boldsymbol{x}^{k} = [1, 2, \textrm{Maint.}]$.}
    \label{fig:calcp}
\end{figure}
\begin{figure}[p]
    \centering
    \includegraphics[scale=0.98]{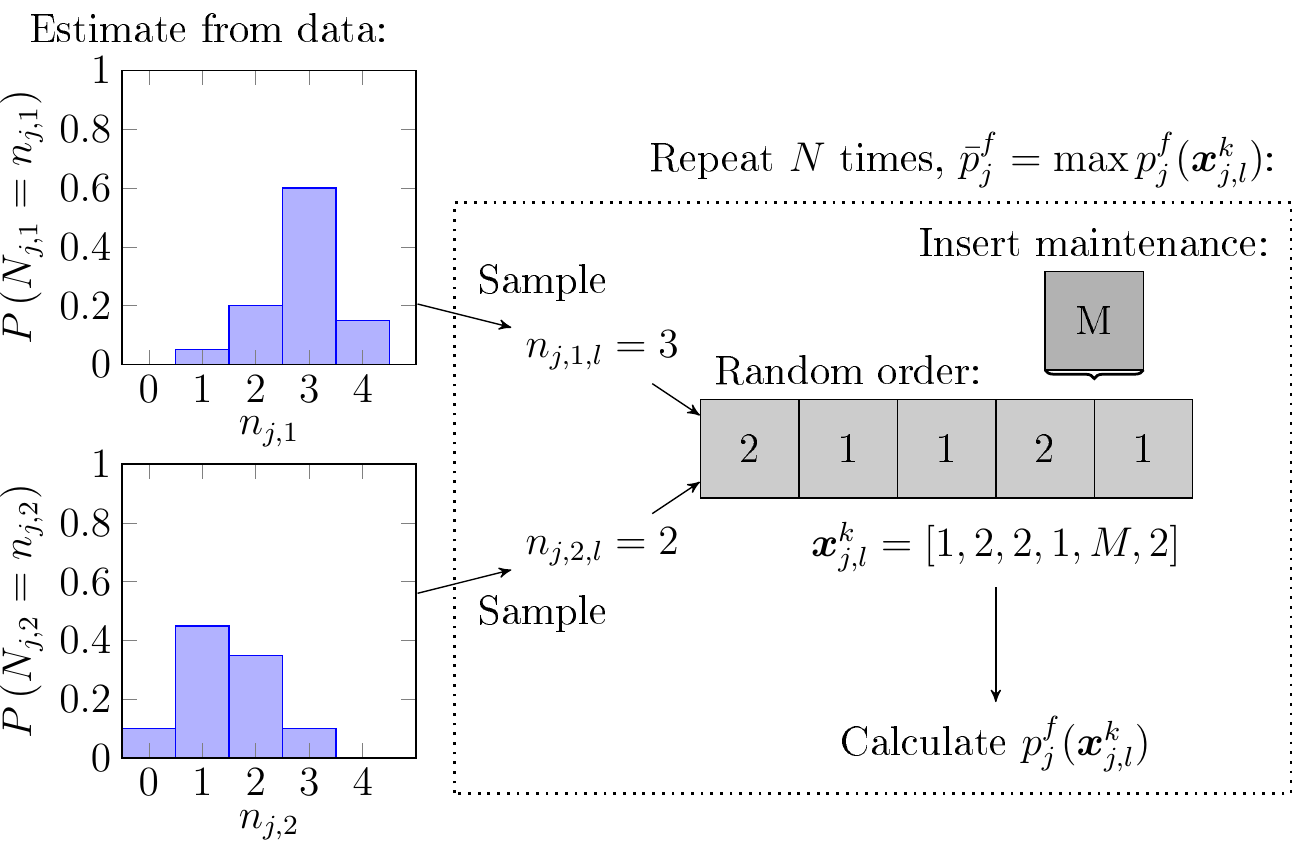}
    \caption{Frequency approach: Estimating $p^f_j$ from historical data using
    Algorithm~\ref{algo:freq}.}
    \label{fig:freq}
\end{figure}

In the special case where $D_{j,k}$ is normal distributed $\mathcal{D}_{j,k}
\sim \mathcal{N}(\mu_{j,k}\Delta t, \sigma^2_{j,k}\Delta t)$, i.e., the
Wiener process model is used, $p^f_j$ can be efficiently calculated using analytical results for the crossing probability of a Brownian motion on a piecewise linear boundary \cite{Poetzelberger1997, Bian2011}.
Instead of sampling from $\mathcal{D}_{j,k}$ at regular time intervals $\Delta t$, this approach only randomly samples at operating mode transitions ($k_t \neq k_{t+1}$).
It requires far less Monte-Carlo samples and is therefore faster than the general method outlined above.
A detailed description of this approach is given in the
Appendix~\ref{sec:apppoetzel}.

\subsection{Estimating failure probabilities}
\label{sec:mc}
Evaluating the probability of failure $p^f_j$, e.g.,using Eqn.~\ref{eq:pfail}, requires the exact sequence of operating modes $\boldsymbol{x}^k_j$ and maintenance actions to be known over the evaluation horizon $T$.
Since maintenance tends to be infrequent, $T$ has to be sufficiently long to obtain meaningful failure probabilities.
Solving Problem~\ref{eq:rob} over a long time horizon may be computationally challenging.
Instead, it may be possible to use existing data of past schedules to estimate $p_j^f$.
If no historical data is available, it can be generated by solving Problem~\ref{eq:rob} over a shorter horizon.
This section outlines two methods by which an upper estimate of $p_j^f$ can be obtained from data.

\subsubsection{Frequency approach}

Assuming time discretization, a conceptually easy way to obtain an upper bound $\bar{p}_j^f$ on $p_j^f$ is to generate the set $\mathcal{X}$ of all possible permutations of operating mode sequences $\boldsymbol{x}^k_j = [k_{j,1}, k_{j,2}, \ldots, k_{j,T}]$ and find the maximum probability of failure
\begin{equation*}
\bar{p}_j^f = \max_{\boldsymbol{x} \in \mathcal{X}} p_j^f(\boldsymbol{x}^k_j).
\end{equation*}
For any realistic problem $\mathcal{X}$ will be very large, but there are two ways to reduce its size:
First, the operating mode sequences can be generated without considering maintenance.
Maintenance actions can then be inserted consecutively at the latest point in time $t_{m,l}$ which satisfies
\begin{equation}
    \label{eq:crit-m}
    \max_{\tilde{d}_{j,k} \in \mathcal{U}} \sum_{t'=t_{m,l-1}}^{t_{m,l}}\sum_k
    x_{j,k,t'}\tilde{d}_{j,k} < s_j^{max} - s_j^0,
\end{equation}
where $t_{m,l-1}$ is the previous maintenance activity and $t_{m,0} = 0$.
$\bar{p}_j^f$ remains an upper bound, because maintenance at a later point in time always causes a larger probability of failure.
Secondly, it may be possible to estimate the frequency of occurrence $n_{j,k} = \sum_t{x_{j,k,t}}$ of each operating mode $k$ from data.
If these frequencies are modeled as random variables $N_{j,k}$, a smaller $\mathcal{X}$ can be obtained by only generating sequences which obey frequencies drawn from the distributions of $N_{j,k}$.
This suggests the following algorithm for obtaining an estimate of $\bar{p}_j^f$ which is also visualized in Fig.~\ref{fig:freq}:
\begin{algorithm}
    \caption{Frequency approach [illustrated in Fig.~\ref{fig:freq}]}
    \label{algo:freq}
    \begin{algorithmic}[1]
        \Procedure{estimate $\bar{p}^f_j$}{}
        \State $\eta_{n_{j,k}} = P(N_{j,k} = n_{j,k}) \gets$ estimate from historical data $\forall n_{k}$
        \State $l \gets 1$
        \While{$l \leq N$}
            \State ${n_{j,k,l}} \gets$ draw random sample from $P(N_{j,k} = n_{j,k})$ for each $k$
            \State $\boldsymbol{x}^k_{j,l} \gets$ arrange $n_{j,l} =\sum_k{n_{j,k,l}}$ op.~modes in random order $[k_1, k_2, \ldots, k_{n_{j,l}}]$
            \State $\boldsymbol{x}^k_{j,l} \gets$ insert maintenance at last possible points in time \Comment{Eqn.~\ref{eq:crit-m}}
            \State $p_{j,l}^f \gets p_{j}^f(\boldsymbol{x}^k_{j,l})$\Comment{Eqn.~\ref{eq:pfail}}
        \EndWhile
        \State $\bar{p}_{j}^f \gets \max_{l \leq N} {p_{j,l}^f}$
        \EndProcedure
    \end{algorithmic}
\end{algorithm}

If $N$ is large enough and the estimated distribution of $N_{j,k}$ is accurate, $\bar{p}_j^f$ should be a good upper bound on $p_j^f$.

\subsubsection{Markov chain approach}
\label{sec:mc-approach}

The second approach for estimating $\bar{p}^f_j$ is inspired by the use of
Markov chains in regime-switching models in finance and to some extent also in
the CBM literature for modeling different enviromental or operating regimes of
a process \cite{Chevallier2017, Breuer2012, Ozekici1995, Li2016}.
The key idea is to treat the occurrence of operating modes $k$ over time as a Markov chain.
Modeling the sequence of operating modes $\boldsymbol{x}^k_j$ on a unit by a memoryless Markov chain $X^k_j(t) = \{X^k_{j,t} : t \leq T\}$, the probability $\pi_{k,k^*}$ of transitioning from one operating mode $k$ to another $k^*$ is given by
\begin{equation*}
    \pi_{k,k^*} = P(X^k_{j,t}=k^*|X^k_{j,t-1}=k).
\end{equation*}
The transition probabilities $\pi_{k,k^*}$ can be estimated
from data.

From this Markov chain random sequences of operating modes $\boldsymbol{x}^k_{j,l}$ can be generated.
Maintenance can again be inserted at the latest possible point in time according to Eqn.~\ref{eq:crit-m}.
$\boldsymbol{x}^k_{j,l}$ may not be a feasible solution to
Problem~\ref{eq:rob}, but it can be used to estimate $\bar{p}_j^f$.
The approach is summarized in Algorithm~\ref{algo:mc}:
\begin{algorithm}
    \caption{Markov chain approach}
    \label{algo:mc}
    \begin{algorithmic}[1]
        \Procedure{estimate $\bar{p}^f_j$}{}
        \State $\pi_{k,k^*} = P(X^k_{j,t}=k^*|X^k_{j,t-1}=k) \gets$ estimate from
        historical data $\forall (k, k^*)$
        \State $l \gets 1$
        \While{$l \leq N$}
            \State $\boldsymbol{x}^k_{j,l} \gets$ draw random operating mode
            sequence from Markov chain $\pi_{k,k^*}$
            \State $\boldsymbol{x}^k_{j,l} \gets$ insert maintenance at last possible points in time \Comment{Eqn.~\ref{eq:crit-m}}
            \State $p_{j,l}^f \gets p_{j}^f(\boldsymbol{x}^k_{j,l})$\Comment{Eqn.~\ref{eq:pfail}}
        \EndWhile
        \State $\bar{p}_{j}^f \gets \max_{l \leq N} {p_{j,l}^f}$
        \EndProcedure
    \end{algorithmic}
\end{algorithm}

\subsubsection{Logistic regression}
\label{sec:lr}
The optimal sequence of operating modes $\boldsymbol{x}^{k,*}_{j}$ depends not only on the structure of the process and the size of the uncertainty set $\mathcal{U}(\alpha)$, but also on parameters $\boldsymbol{\psi}(t)$ such as product demands or environmental variables.
The distributions of $N_{k}$ and $\pi_{k,k^*}$ are therefore not necessarily stationary:
\begin{subequations}
\begin{align}
    \eta_{n_{j,k}}\left(\boldsymbol{\psi}(t)\right)
    &= P(N_{j,k} = n_{j,k}|\boldsymbol{\psi})\\
    \pi_{k,k^*}\left(\boldsymbol{\psi}(t)\right)
    &= P(X^k_{j,t}=k^*|X^k_{j,t-1}=k,\boldsymbol{\psi}),
\end{align}
\end{subequations}
where $\eta_{n_{j,k}}(\boldsymbol{\psi})$ is the probability that operating mode $k$ occurs $n_{j,k}$ times in time period $\Delta t$.

\begin{figure}[htbp]
    \centering
    \includegraphics{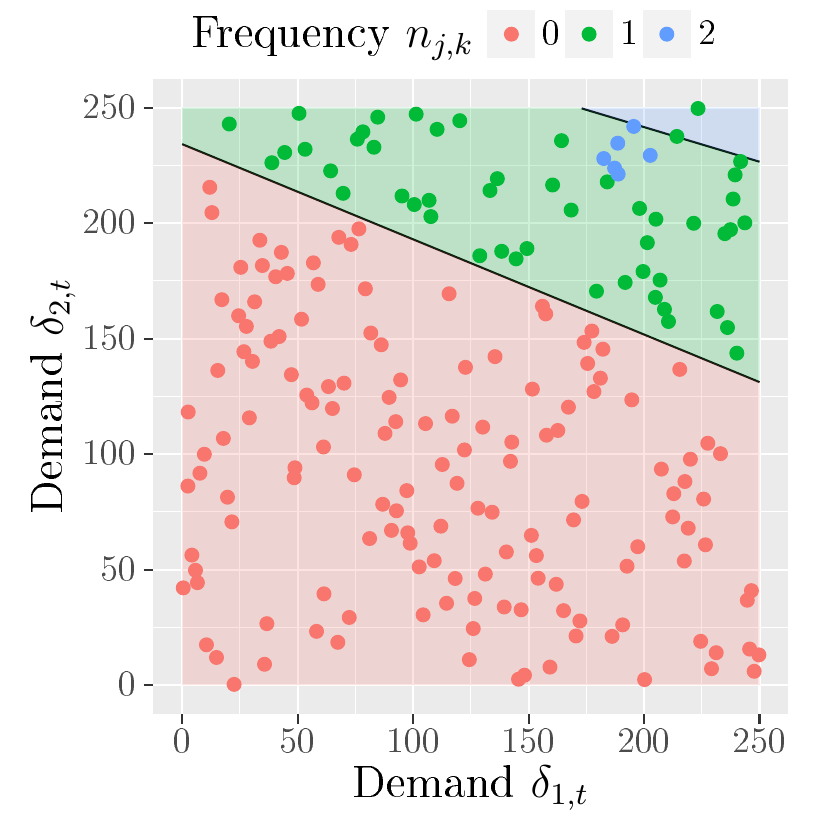}
    \caption{Frequency $n_{j,k}$ of operating mode $k$ occuring on unit $j$ for
    a scheduling problem with two product demands $\boldsymbol{\psi} =
[\delta_{1,t}, \delta_{2,t}]^{\top}$. Points are training data generated by
solving the scheduling model and shaded areas are predictions by logistic
regression.}
    \label{img:logreg-toy}
\end{figure}
Covariate dependency of Markov chain transition probabilities has previously been modeled by using logistic regression \cite{Paton2014,Sinha2011}.
We model both $\eta_{n_{j,k}}$ and $\pi_{k,k^*}$ using multinomial logistic regression in order to capture the influence of product demands:
\begin{subequations}
\begin{align}
    \label{eq:logreg}
    \eta_{n_{j,k}}(\boldsymbol{\psi}(t)) = &
    \frac{\exp(\boldsymbol{\beta^{\top}}_{n_{j,k}}\boldsymbol{\psi})}{\sum_{n'_{k}}
    \exp(\boldsymbol{\beta^{\top}}_{n'_{j,k}}\boldsymbol{\psi})}\\
    \pi_{k,k^*}(\boldsymbol{\psi}(t)) =           &
    \frac{\exp(\boldsymbol{\beta^{\top}}_{k,k^*}\boldsymbol{\psi})}{\sum_{k^+ \in
    K}\exp(\boldsymbol{\beta^{\top}}_{k,k^+}\boldsymbol{\psi})}.
\end{align}
\end{subequations}
We use Scikit-learn\cite{sklearn} for estimating parameters
$\boldsymbol{\beta}$ based on data.
Fig.~\ref{img:logreg-toy} shows an example for a process with two product demands $\boldsymbol{\psi} = [\delta_{1,t}, \delta_{2,t}]^{\top}$.
The shaded areas are the frequencies $n_{j,k}$ predicted by logistic regression for a particular $k$ and $j$ (the $n_{j,k}$ with the largest $\eta_{n_{j,k}}(\boldsymbol{\psi})$) while the points are training data.
We use logistic regression in this work because of its simplicity and interpretability, but it could be replaced by any classification method capable of probability estimation, e.g., Artificial Neural Networks, Support Vector Machines, k-Nearest Neighbours, Decision Trees, etc\cite{Bravo2010, Dreiseitl2002}.

\section{Optimizing the uncertainty set size}
An important, non-trivial decision when using robust optimization is the size
of the uncertainty set --- or in this work the choice of parameter $\alpha$.
It governs a trade-off between the robustness of the solution and its cost.
A common approach is to use a-priori guarantees to determine an uncertainty set size that is guaranteed to have a probability of constraint violation below a predefined level \cite{}.
A-priori guarantees are, however, not guaranteed to be tight and uncertainty sets based on them can be overly conservative.
As demonstrated by Li and Li \cite{Li2015, Li2015_2}, determining the optimal uncertainty set size can instead be seen as its own optimization problem.
They minimize the uncertainty set size with the constraint that the solution remains feasible with a pre-defined probability.
We propose a different formulation that does not require the decision maker to choose a probability of constraint satisfaction but is based purely on cost instead:
\begin{equation}
    \label{eq:optalpha}
    \min_{\alpha} c^*(\alpha) + \sum_{j} p^f_{j}(\alpha)\cdot c_{j}^f,
\end{equation}
where $c^*(\alpha)$ is the minimal overall cost of the process as determined by
solving Problem~\ref{eq:rob} for a given value of $\alpha$, $p^f_j$ is the
corresponding probability of failure evaluated using Eqn.~\ref{eq:pfail}, and
$c^f_j$ is the cost incurred in case of an unplanned failure of unit $j$,
i.e., the cost of corrective maintenance.
Effectively, Problem~\ref{eq:optalpha} minimizes the trade-off between preventive and corrective maintenance.
Note that this formulation assumes that each unit fails no more than once in the evaluated horizon $T$.
This is reasonable under the assumption that the cost of failure $c^f_j$ is high and therefore $P^f_j$ tends to be low.

Problem~\ref{eq:optalpha} is a one-dimensional optimization problem, but determining $c^*(\alpha)$ and $p^f_j(\alpha)$ can be computationally expensive because it requires solving a potentially large MILP problem and Monte-Carlo simulation.
It can therefore be viewed as a black box optimization problem with expensive function evaluations.
We propose Bayesian optimization, which is known to work well on expensive low dimensional objective functions, as an effective solution strategy \cite{Jones1998}.
Bayesian optimization has the further advantage that it can handle noise well.
Both $c^*$ and $p^f_j$ can be noisy because it may not be possible to solve Problem~\ref{eq:rob} to optimality in a reasonable time frame.
Further noise is introduced by the Monte-Carlo simulation used to evaluate $p^f_j$.

\section{Case study}

\begin{figure}[htbp]
    \centering
    \includegraphics[width=\textwidth]{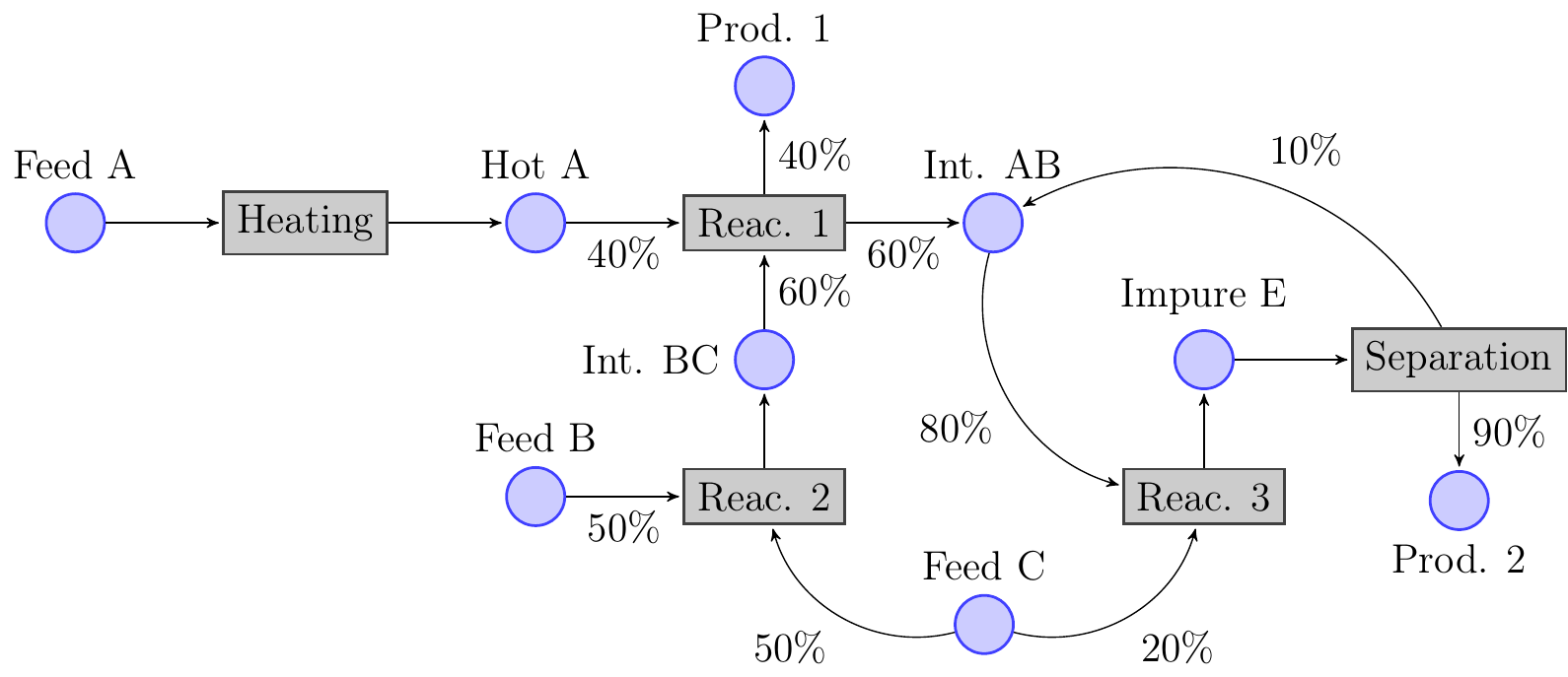}
    \caption{STN instance proposed by \citet{Kondili1993}. Tasks are performed
    on four units: Heater, Reactor~1, Reactor~2, and Still.}
    \label{fig:kondili}
\end{figure}
The model by \citet{Biondi2017}, an extension of the State-Task-Network (STN) \cite{Kondili1993}, forms the basis of our case study.
The classic STN is a scheduling problem in which a set of tasks $I$ has to be assigned to a set of units $J$.
\citet{Biondi2017} extend the STN by allowing each task $i$ to be performed in a number of different operating modes $k \in K_i$.
They add constraints reducing the residual lifetime $r_{j,t}$ of each unit $j$ every time a task is performed, and restore $r_{j,t}$ by performing maintenance.
Because the scheduling problem can only be solved for a short time horizon $T_S$ but maintenance occurs infrequently, they add a planning horizon $T_P$ to the problem.
For the planning horizon, instead of an exact schedule, only the number of times $n_{i,j,k,t}$ a task $i$ is performed on unit $j$ in operating mode $k$ in each planning period $t$ is calculated.
Eqns.~\ref{stn:obj} to~\ref{stn:last} give the modified version used as a case study in this work:

\allowdisplaybreaks
Objective function:
\begin{align}
    \begin{split}
    \text{cost} = \sum_{j \in J} c_j^{maint}\left(s_j^{fin}/s_j^{max} + \sum_{t \in
    T} m_{j,t}\right)\\ + c_s^{storage}\left(q_s^{fin} + \sum_{t \in T_p}
    q_{s,t}\right) \\
    + U\left(\sum_{s \in S} \phi^d_s + \sum_{t \in T_S} \phi^q_{s,t}\right)
    \end{split} \label{stn:obj}
\end{align}

Constraints scheduling horizon:
\begin{subequations}
\begin{align}
    & \sum_{k \in K_j} \sum_{i\in I_j} \sum_{t'=t-p_{i,j,k}+\Delta t_S}^t w_{i,j,k,t'} + \sum_{t'=t-\tau_j+\Delta t_s}^t m_{j,t'} \leq 1
    &
    & \forall J, t \in T_S                                                    \\
    & v^{min}_{i,j}w_{i,j,k,t} \leq b_{i,j,k,t} \leq v^{max}_{i,j}w_{i,j,k,t}
    &
    & \forall J, i \in I_j, k \in K_j, t \in T_S                            \\
    \begin{split}
        q_{s,t} = q_{s,t-1} + \sum_{i \in \bar{I}_s} \bar{\rho}_{i,s}\sum_{j \in
        J_i}\sum_{k \in K_j} b_{i,j,k,t-p_{i,j,k}}                    \\
        - \sum_{i \in I_s}\rho_{i,s}\sum_{j\in J_i}\sum_{k\in K_j}b_{i,j,k,t}
    \end{split}
    &
    & \forall s, t \in T_S                                                    \\ \label{eq:slack1}
    & 0 \leq q_{s,t} - \phi^q_{s,t}\leq c_s
    &
    & \forall s, t \in T_S                                                    \\
    & m_{j,t} s_{j}^{0} \leq s_{j,t} \leq s_{j}^{max} + m_{j,t} \cdot (s_{j}^{0} -
    s_{j}^{max})
    &
    & \forall t, j \in J, D \in \mathcal{U}                                    \\
    & s_{j,t} \geq s_{j,t-\Delta t_S} + \sum_{i}\sum_{k}{w_{i,j,k,t}\tilde{d}_{j,k}}
    + m_{j,t} \cdot (s_{j}^{0} - s_{j}^{max})
    &
    & \forall t, j \in J, D \in
    \mathcal{U}                                            \\
    & s_{j,t} \leq s_{j,t-\Delta t_S} + \sum_{i}\sum_{k}{w_{i,j,k,t}\tilde{d}_{j,k}}
    &
    & \forall t, j \in J, D \in \mathcal{U},
\end{align}
\end{subequations}

Constraints planning horizon:
\begin{subequations}
\begin{align}
    &\sum_{i \in I_j}\sum_{k \in K_j} p_{i,j,k} n_{i,j,k,t} + \tau_j m_{j,t} \leq \Delta t_P
    &
    & \forall J, t \in T_{P}\backslash{\{\bar{t}_P\}} \\
    &v_{i,j}^{min}\sum_{k\in K_j} n_{i,j,k,t} \leq a_{i,j,t} \leq
    v_{i,j}^{max}\sum_{k\in K_j} n_{i,j,k,t}
    &
    & \forall J, i \in I_j, k \in
    K_j, t \in T_P\\
    &q_{s,t} = q_{s,t-1} + \sum_{i \in \bar{I}_s}\bar{\rho}_{i,s} \sum_{j \in
    J_i}a_{i,j,t} - \sum_{i \in I_s} \rho_{i,s} \sum_{j \in J_i} a_{i,j,t}
    - \delta_{s,t}
    &
    & \forall s, t \in T_P \backslash\{\bar{t}_P\}    \\
    & 0 \leq q_{s,t} \leq c_s
    &
    & \forall s, t \in T_P                            \\
    & n_{i,j,k,t} \leq U \cdot \omega_{j,k,t}
    &
    & \forall J, i \in I_j, k \in K_j, t \in T_P      \\
    &\sum_{k \in K_j} \omega_{j,k,t} = 1
    &
    & \forall J, t \in T_P                            \\
    & s_{j,t} \leq s_{j}^{max}
    &
    & \forall t, j \in J                              \\
    & s_{j}^{t} \geq s_{j,t-\Delta t_P} + \sum_{k}{n_{j,k,t}\tilde{d}_{j,k,t}} +
    m_{j,t} \cdot (s_{j}^{0} - s_{j}^{max})
    &
    & \forall t, j \in J                              \\
    & s_{j,t} \leq s_{j,t-\Delta t_P} + \sum_{k}{n_{j,k,t}\tilde{d}_{j,k,t}}
    &
    & \forall t, j \in J
\end{align}
\end{subequations}

Constraints interface between scheduling and planning:
\begin{subequations}
\begin{align}
    \begin{split}\label{eq:int1}
        \sum_{k \in K_j} \sum_{i \in I_j} \sum_{t'=\bar{t}_S+2\Delta t_S-p_{i,j,k}}^{\bar{t}_S} w_{i,j,k,t'} \left[p_{i,j,k} - (\bar{t}_S - t' + \Delta t_S)\right]    \\
        + \sum_{t'=\bar{t}_S+2\Delta t_S-\tau_j}^{\bar{t}_S} m_{j,t'}\left[\tau_{j} - (\bar{t}_S - t' + \Delta t_S)\right]    \\
        + \sum_{i \in I_j} \sum_{k \in K_j} p_{i,j,k}n_{i,j,k,\bar{t}_P} + \tau_jm_{j,\bar{t}_P} \leq \Delta t_P,
    \end{split}
    &
    & \forall j \in J                                       \\
    \begin{split}\label{eq:slack2}
    q^{fin}_s = q_{s,\bar{t}_S} + \sum_{i\in \bar{I}_s} \bar{\rho}_{i,s}\sum_{j \in
    J_i}\sum_{k \in K_j}b_{i,j,k,\bar{t}_S+1-p_{i,j,k}}\\ - \delta_{s,\bar{t}_S}
    + \phi^d_{s}
    \end{split}
    &
    & \forall s                                       \\
    & 0 \leq q^{fin}_s \leq c_s
    &
    & \forall s                                       \\
    \begin{split}q_{s,\bar{t}_P} = q^{fin}_s + \sum_{i \in \bar{I}_S}
    \bar{\rho}_{i,s} \sum_{j \in J_i} \sum_{k \in K_j} \sum_{t'=\bar{t}_s +
    2 - p_{i,j,k}}^{\bar{t}_S} b_{i,j,k,t'} \\ + \sum_{i \in \bar{I}_S}
    \bar{\rho}_{i,s}\sum_{j,J_i}a_{i,j,\bar{t}_P}\\
    - \sum_{i,I_s}\rho_{i,s}\sum_{j \in J_i}a_{i,j,\bar{t}_P} - \delta_{s,\bar{t}_P}
    \end{split}
    &
    & \forall s              \label{stn:last}
\end{align}
\end{subequations}
The decision variables are $m_{j,t}, q_{s,t}, w_{i,j,k,t}, n_{i,j,k,t},
b_{i,j,k,t}, a_{i,j,t}, s_{j,t}, \phi^q_{s,t}, \phi^d_s$ and $\omega_{j,k,t}$.
The product demand $\delta_{s,t}$ has to be
    satisfied at the end of each planning period and at the end of any time interval in the scheduling horizon.
In practice, this model would be solved regularly in a rolling horizon fashion
using recent demand estimates and degradation signal measurements $s_{j,0}$.

In comparison to \citet{Biondi2017}, the residual lifetime constraints have been replaced with the degradation signal based health model developed above (Eqn.~\ref{eq:rob}).
For the planning horizon, $s_{j,t}$ cannot be reset exactly to $s_j^0$, because the exact time at which maintenance is performed is unknown.
Instead, it is merely enforced that $s_{j,t} \leq s_{j}^{max}$.

In addition, the objective function is slightly different.
The term $c_j^{maint}(s_j^{fin}/s_j^{max})$ can be interpreted as a final cost of maintenance dependent on the final degradation signal $s_j^{fin}$ (state of health) of unit $j$.
Similar to Biondi et al.'s\cite{Biondi2017} penalty terms it avoids unnecessary degradation and ensures maintenance happens towards the end of a units residual lifetime.

Since the exact sequence of operating modes and maintenance actions is unknown in the planning horizon $T_P$, the probability of failure $p^f_j$ can only be evaluated over the scheduling horizon $T_S$.
In order to still evaluate $p^f_j$ over a longer time period two possibilities exist:
the schedule can be extended in length by solving the model repeatedly in a rolling horizon fashion or the Markov chain-based estimation approach in Algorithm \ref{algo:mc} can be used.
We compare both approaches to show that the proposed Markov chain estimate is indeed accurate.
Note that, in order to facilitate a rolling horizon based solution approach, slack variables have been introduced in Eqns.~\ref{eq:slack1} and~\ref{eq:slack2}.
This is necessary because the rolling horizon framework does not guarantee feasibility in subsequent time periods. Production targets from the planning model may, for example, not be achievable in the scheduling model.
The slack variables $\phi^d_{s}$ and $\phi^q_{s,t}$ are penalized in the objective function.

We assume that the frequencies of operating mode occurrence $n_{j,k}$ and the transition probabilities $\pi_{k,k*}$ depend on the product demands in each planning period ($\boldsymbol{\phi}_t = [d_{s_1,t},\ldots,d_{s_n,t}]$).
We sample a range of demands using Latin Hypercube Sampling and solve just the scheduling horizon to generate data for estimating $n_{j,k}\left(\boldsymbol{\psi}_t\right)$ and $\pi_{k,k*}\left(\boldsymbol{\psi}_t\right)$.

Notice that the model above fulfills the assumptions in Theorem~\ref{theo:rob-det} and solutions can be obtained by solving the deterministic approximation~\ref{eq:det}.

\section{Results}
The framework outlined above was evaluated on five instances of the STN (see
Table~\ref{tab:instances} and Appendix~\ref{sec:appc}).
\begin{table}[htb]
    \centering
    \begin{tabular}{l r r r r r}
        \hline
        Instance            & Toy  & P1\cite{Kondili1993}   & P2\cite{Karimi1997}   & P4\cite{Maravelias2003}   & P6\cite{Ierapetritou1998}   \\ \hline
        Units               & 2    & 4    & 5    & 3    & 6    \\
        Tasks               & 3    & 5    & 3    & 4    & 8    \\
        Op.~modes           & 2    & 3    & 3    & 2    & 2    \\
        Products            & 2    & 2    & 1    & 2    & 4    \\ \hline
        Discrete vars       & 518  & 2492 & 1930 & 1869 & 1993 \\
        Continuous vars     & 1033 & 3630 & 2371 & 2777 & 4084 \\
        Constraints         & 1860 & 7332 & 5705 & 5699 & 7994 \\
        Avg.~MIP gap [\%]   & 0.0  & 3.0  & 5.8  & 10.9 & 1.02 \\
        \hline
    \end{tabular}
    \caption{Evaluated STN instances (Details: see Appendix~\ref{sec:appc})}
    \label{tab:instances}
\end{table}
The model was implemented in Pyomo\cite{Hart2017, Pyomo2011} and solved using CPLEX 12.7.1.0.
All source code is publicly available under the MIT Licence\cite{github}.
Unless mentioned otherwise, we considered an evaluation horizon of $12$ planning periods.
The failure probability $p^f_j$ for each unit $j$ was evaluated for a range of values of the uncertainty set parameter $\alpha$ using both the frequency and Markov chain estimates as well as rolling horizon.
The termination criteria for each CPLEX run were a maximum time limit between $1-5$ minutes (depending on the size of the instance) and a MIP gap of $2\%$ (except for the toy instance which was solved to optimality).
For each instance a low, average, and high demand scenario was considered with
the high scenario being close to maximum process capacity.
For Instance~P1\cite{Kondili1993} both the robust and deterministic Problems~\ref{eq:rob} and~\ref{eq:det} were solved a number of times to evaluate the quality of the deterministic approximation.
For all other instances only the deterministic approach was used.
All calculations were carried on an i7-6700 CPU with $8 \times 3.4\textrm{GHz}$ and 16GB RAM.

\begin{figure}[htbp]
    \centering
    \includegraphics{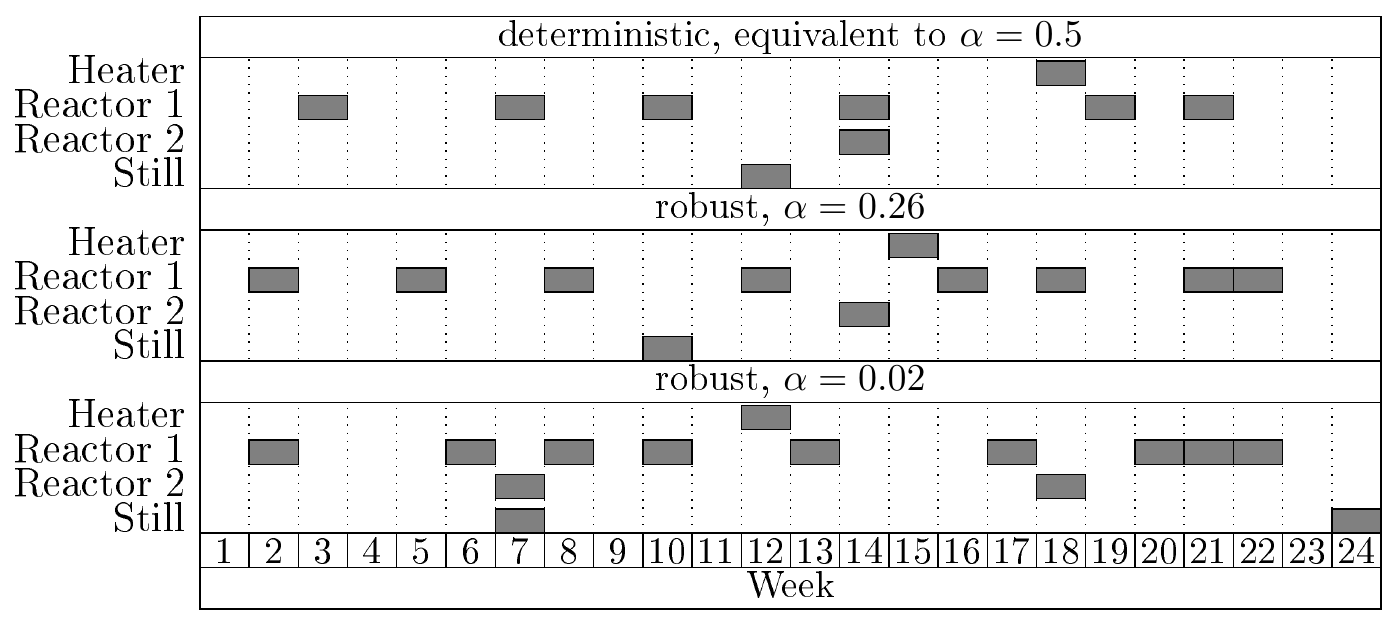}
    \caption{Comparison of maintenance schedules between deterministic solution
        ($\alpha = 0.5$) and two robust solutions with different values of
    $\alpha$ (Instance~P1\cite{Kondili1993}, average demand scenario
by Biondi et al.\cite{Biondi2017}). The number of required maintenance actions
increases with increasing uncertainty set size (decreasing $\alpha$).}
    \label{fig:gantt}
\end{figure}
    Fig.~\ref{fig:gantt} shows three maintenance schedules for the original STN
    instance~(P1) by \citet{Kondili1993} with Biondi et al.'s\cite{Biondi2017}
    demand scenario (average scenario) with different values for $\alpha$.
    The number of maintenance actions increases with
    increasing uncertainty set size (decreasing $\alpha$).
    Essentially, hedging against more uncertainty and ensuring solution
    robustness for a larger set of possible realizations requires earlier
    maintenance. In the average demand scenario, maintenance actions increase
    by $22\%$ when hedging against some of the uncertainty ($\alpha = 0.26$)
    and by $56\%$ when hedging against almost all uncertainty ($\alpha = 0.02$).
    However, this trend also depends on product demand: for the low demand
    scenario, only an increase of $25\%$ is necessary for $\alpha = 0.02$,
    while an increase of $64\%$ is necessary in the high demand scenario. A
    higher demand increases unit utilization and therefore also the absolute
    number of maintenance actions required in a given time period.

\begin{figure}[htbp]
    \centering
    \includegraphics[width=\textwidth]{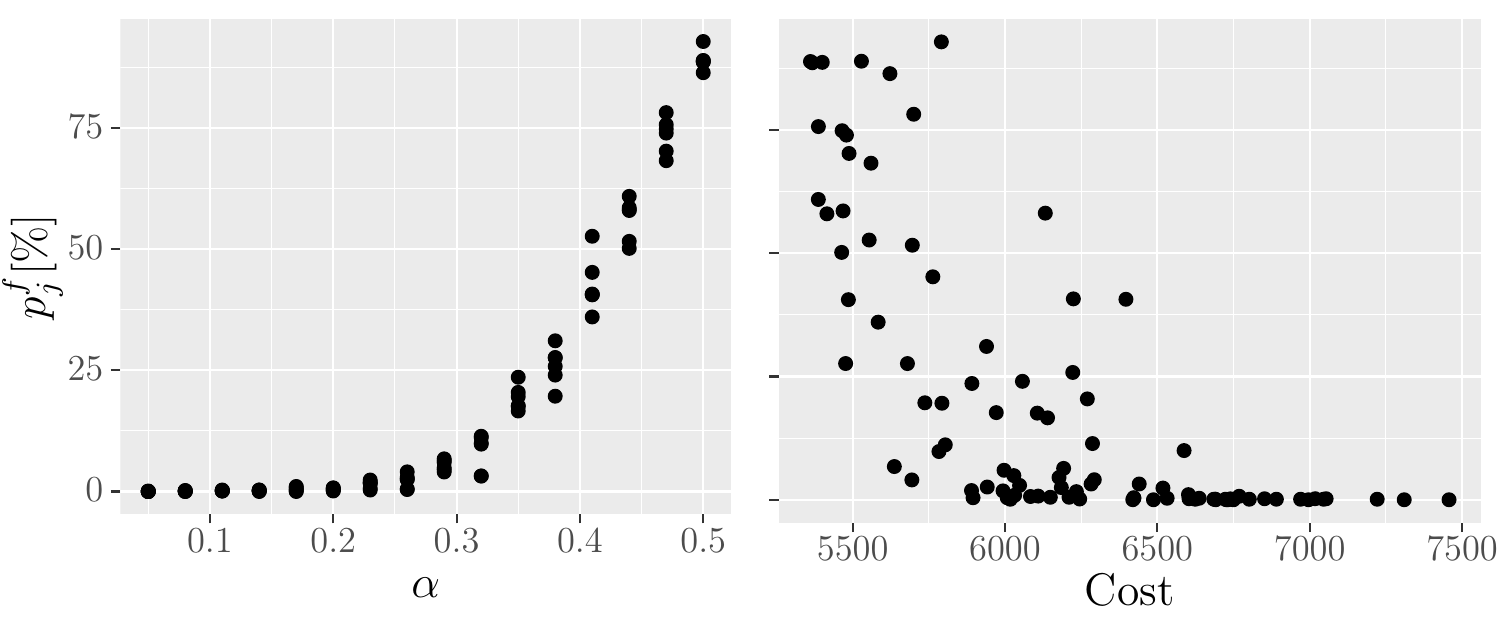}
    \caption{Probability of Reactor 1 failing $p^f_j$ vs uncertainty set
        parameter $\alpha$ and cost
        (Instance~P1\cite{Kondili1993}, average demand scenario). Each point is a solution to
        Problem~\ref{eq:det} obtained by CPLEX. The probability of failure
        $p^f_j$ was estimated using the Algorithm \ref{algo:mc} Markov chain approach.} \label{fig:p-vs-cost-R1}
\end{figure}
\begin{figure}[htbp]
    \centering
    \includegraphics[width=\textwidth]{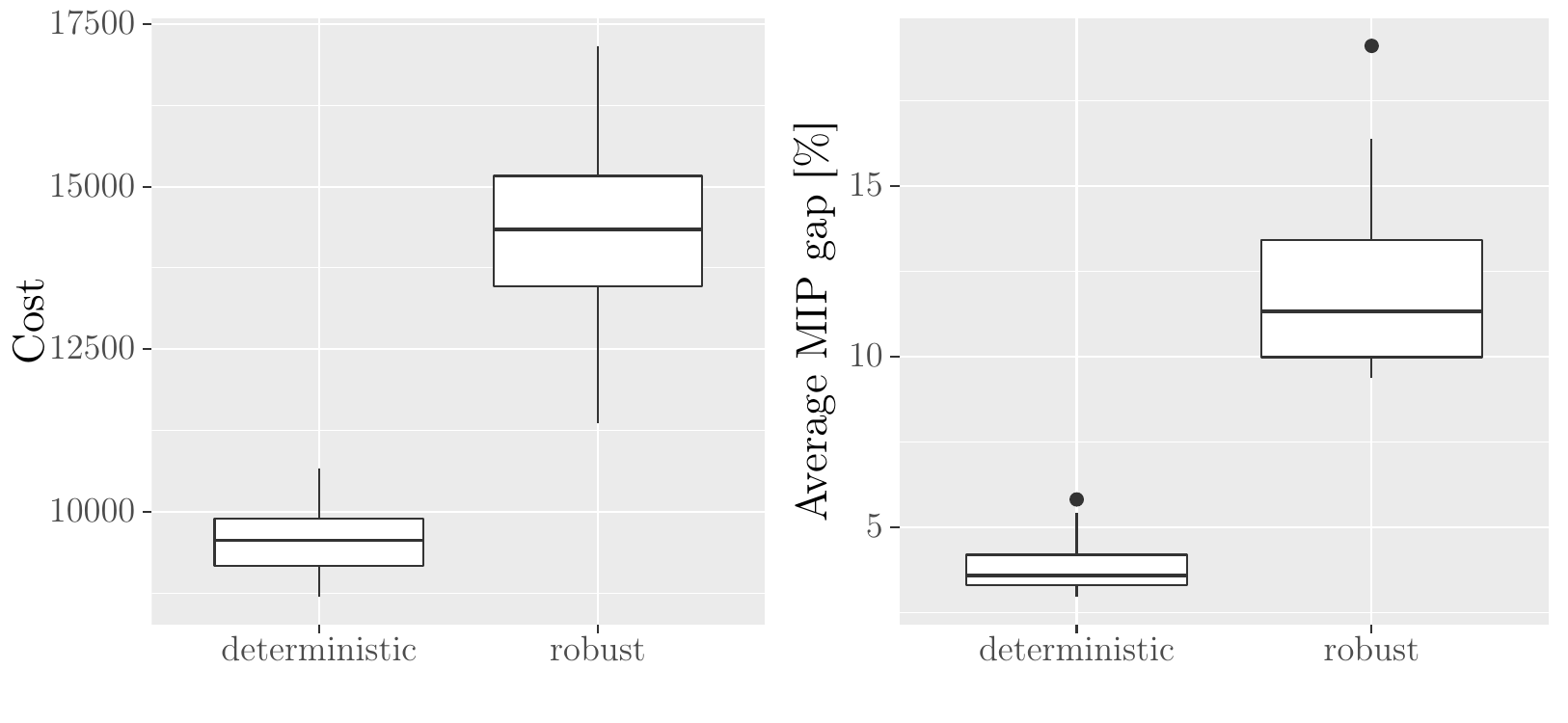}
    \caption{Robust vs.~deterministic approach (Instance~P1\cite{Kondili1993}).
    Each model was solved $25$ times with $\alpha=0.41$ and a $120s$ time limit
    per rolling horizon iteration
using out-of-the-box CPLEX. The robust model has 22514 variables and 14664
constraints while the deterministic model has 6122 variables and 7332
constraints.}
    \label{fig:det-vs-rob}
\end{figure}
Fig.~\ref{fig:p-vs-cost-R1} shows the failure probability $p^f_j$ for Reactor 1 as a function of both total cost (cost of storage and cost of maintenance) and the uncertainty set parameter $\alpha$.
As expected, $p^f_j$ increases for smaller uncertainty sets (large $\alpha$'s) and a low $p^f_j$ comes at a significant cost -- the price of robustness.
Notice that, while calculating $p^f_j$ using Monte-Carlo simulation only introduces modest noise, the calculated cost is very noisy due to the non-optimality of the solutions.

The results in Fig.~\ref{fig:p-vs-cost-R1} were obtained by solving the deterministic Problem~\ref{eq:det}.
While Theorem~\ref{theo:rob-det} guarantees that these solutions are also feasible in the robust Problem~\ref{eq:rob}, it does not prove that they are also optimal.
Fig.~\ref{fig:det-vs-rob} shows both total cost and average optimality gap for a number of rolling horizon solutions to the deterministic and robust version of {Instance~P1\cite{Kondili1993}}.
The uncertainty set size was $\alpha=0.41$ and a maximum time limit of $120s$ was used for each CPLEX run.
Within this time limit, out-of-the-box CPLEX achieves an average optimality gap of $12.1\%$ on the robust problem compared with $3.8\%$ for the deterministic approximation with maximum values for $\tilde{d}_{j,k}$.
Similarly, the deterministic approximation achieves significantly lower objective values.
While many approaches could improve solution quality of the robust problem,
e.g., solver parameter tuning or leveraging Satisfiability Modulo
Theory\cite{Mistry2018}, Constraint Programming\cite{Cire2016}, or
Approximation Algorithms\cite{Letsios2018}, it is likely that the deterministic approximation will remain favorable as it has significantly fewer variables and constraints ($6122$ vs.~$22514$ and $7332$ vs.~$14664$ respectively).
Assuming a box uncertainty set, we view it as a reasonable approximation for instances which cannot
be solved to optimality in a reasonable amount of time.
In the case of a more complex uncertainty
set, replacing $\tilde{d}_{j,k}$ with its maximum value may lead to
conservative solutions. General uncertainty sets require solving the robust problem.

Note that the large range of solution values in Fig.~\ref{fig:det-vs-rob} is not only due to the differing MIP gaps, but also the rolling horizon approach which does not guarantee optimality.

\begin{figure}[p]
    \centering
    \includegraphics[scale=.96]{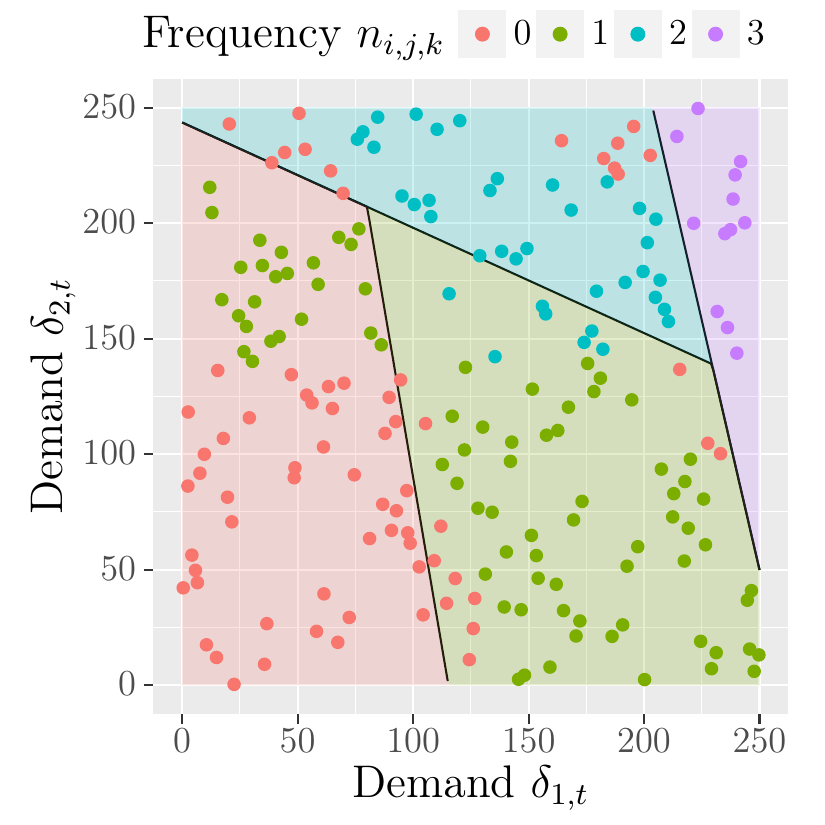}
    \caption{Toy instance: Frequency $n_{i,j,k}$ of Reaction 1 occuring in
    normal mode on unit Reactor within one scheduling horizon. Points are
    training data generated by solving the
scheduling model repeatedly and shaded areas are predictions by logistic
regression.}
    \label{fig:logreg-toy-R-R1N}
\end{figure}
\begin{figure}[p]
    \centering
    \includegraphics[width=0.96\textwidth]{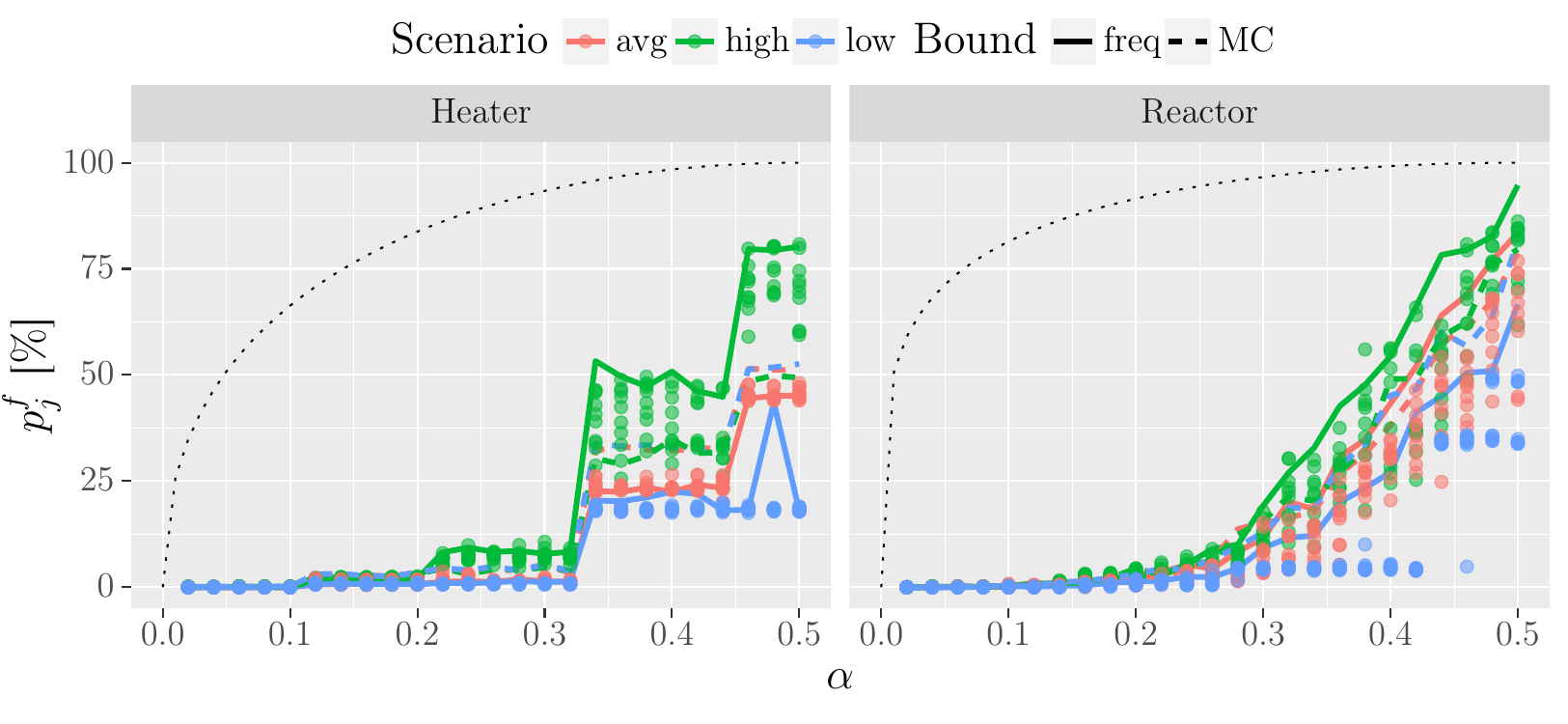}
    \caption {Toy instance: Frequency vs.~Markov chain approach.
    Points are rolling horizon solutions. Colored lines are bounds from the
    Frequency and Markov chain approach.
    The dotted black lines show a-priori bound B4 by \citet{Li2012}.}
    \label{fig:toy2-mc-vs-freq}
\end{figure}
Logistic regressions for the Frequency and Monte-Carlo approaches were trained based on $200$ scheduling horizon only solutions.
For the Frequency approach, the training points and their predicted values for Reaction $1$ in mode Normal of the toy instance are shown in Fig.~\ref{fig:logreg-toy-R-R1N}.
Logistic regression predicts $n_{i,j,k}$ reasonably well but the rigid, linear classifier cannot capture some of the details.
This is, however, not a major problem as the entire predicted probability distribution $\eta_{n_{i,j,k}}$ of operating mode occurence frequencies $n_{i,j,k}$ is used in estimating $\bar{p}^f_j$.
Near the predicted boundaries, $\eta_{n_{i,j,k}}$ of adjacent frequencies will be non-zero and they will be sampled in a significant number of operating mode sequences in Algorithm~\ref{algo:freq}.

Fig.~\ref{fig:toy2-mc-vs-freq} shows the probability of failure $p^f_j$ for each unit in the toy instance as a function of the uncertainty set parameter $\alpha$ for three different demand scenarios (average, high, and low).
Since the rolling horizon framework does not guarantee optimality, solving the problem repeatedly for the same value of $\alpha$ can lead to different solutions and failure probabilities.
The problem was therefore solved $10$ times for each value of $\alpha$.
It can be seen that the probability of failure generally increases with demand.
The figure furthermore shows the two bounds obtained using the Frequency and Markov chain based approaches, i.e.\ Algorithm \ref{algo:freq} versus \ref{algo:mc}.
For the reactor both approaches provide good upper bounds.
For the heater, the frequency based approach performs very well, while the Markov chain approach underestimates $p^f_j$ for the high demand scenario and overestimates it for the average and low demand scenarios.
Finally notice that the apriori bound\cite{Li2012} given by the dotted lines greatly overestimates $p^f_j$.
Fig.~\ref{fig:biondi-mc-vs-freq} shows similar trends for Instance~P1\cite{Kondili1993} (Kondili).
Both approaches provide reasonable bounds for all units and scenarios except the average demand scenario on the Heater, for which the frequency approach underestimates $p^f_j$.
Notice that $p^f_j$ is nearly zero for both the Heater and Reactor 2 at low demand irrespective of $\alpha$.
This is because for this scenario no maintenance occurs on either unit and $s_{j,t}$ does not get close to $s^{max}_{j}$.

\begin{figure}[htbp]
    \centering
    \includegraphics[width=\textwidth]{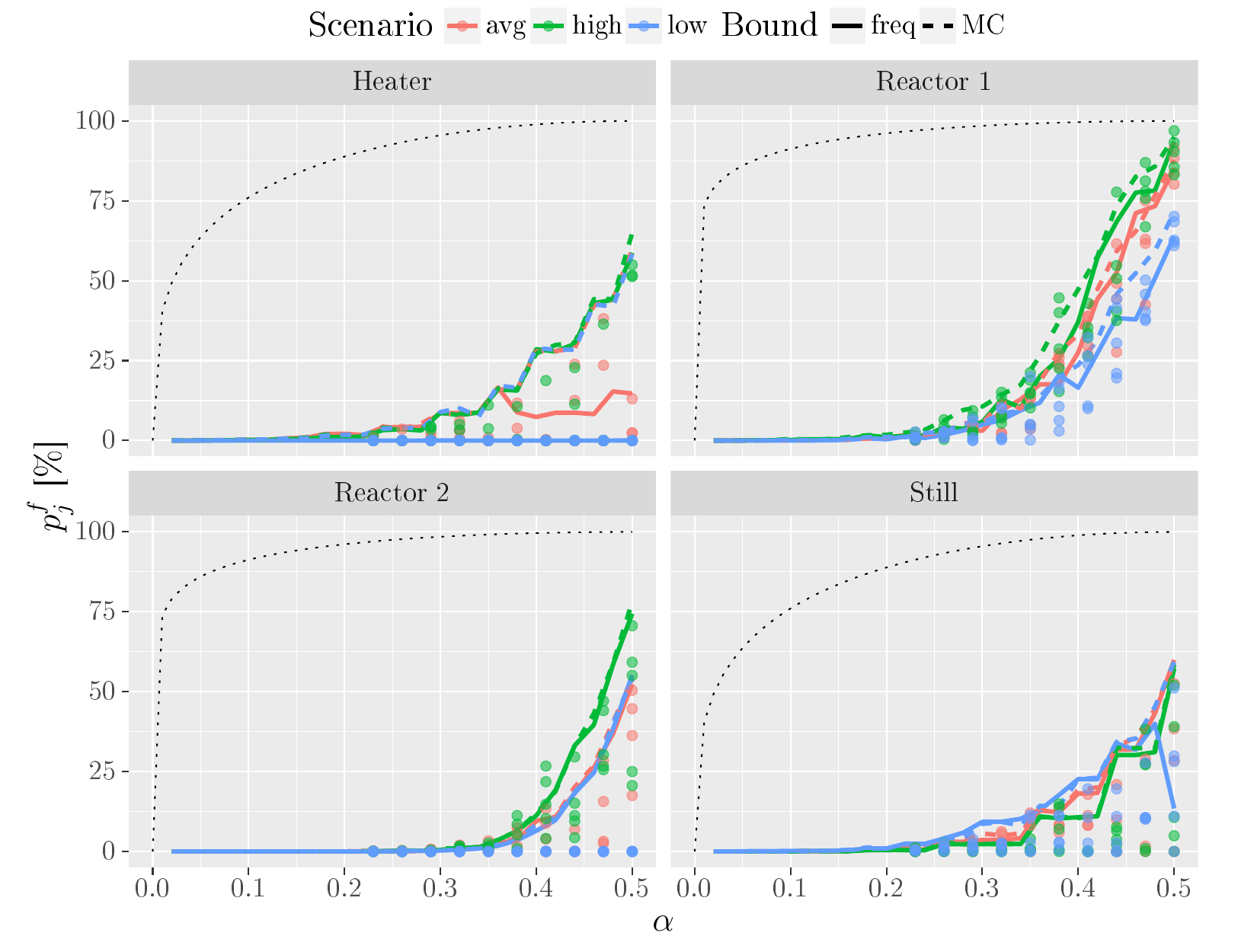}
    \caption{{Instance~P1\cite{Kondili1993}}:
        Frequency vs.~Markov chain approach. Points are rolling horizon solutions. Colored lines are bounds from the Frequency and Markov chain approach.
    The dotted black lines show a-priori bound B4 by \citet{Li2012}.} \label{fig:biondi-mc-vs-freq} \end{figure}

The performance of the Frequency and Markov chain based probability estimates was assessed using three metrics:
\begin{subequations}
\begin{align}
    \textrm{rms}^2_{all} &= \frac{1}{N\cdot|A|}\sum_{n \in \{1..N\}, \alpha
    \in A} \left(\left[p^f_j\right]_{n, \alpha} - \bar{p}^f_j\right)^2,\\
    p_{out} &= \frac{1}{N\cdot|A|}\sum_{n \in \{1..N\}, \alpha \in A}
    \mathbbm{1}\left(\left[p^f_j\right]_{n, \alpha} > \bar{p}^f_j\right)\label{eq:p-out},\textrm{~and}\\
    \textrm{rms}^2_{out}
    &= \frac{1}{p_{out} \cdot N \cdot |A|}
    \sum_{n \in \{1..N\}, \alpha \in A}
    \mathbbm{1}\left( \left[ p^f_j \right]_{n, \alpha} > \bar{p}^f_j \right)
    \left( \left[ p^f_j \right]_{n, \alpha} - \bar{p}^f_j \right)^2,
\end{align}
\end{subequations}
where $\textrm{rms}_{all}$ is the root-mean-squared deviation between the estimate and all rolling horizon solutions, $p_{out}$ is the percentage of rolling horizon solutions with a larger $p^f_j$ than the estimated bound, and $\textrm{rms}_{out}$ is the root-mean-squared deviation of all underestimated points.
While $\textrm{rms}_{all}$ evaluates the estimate $\bar{p}^f_j$'s quality as a predictor of $p^f_j$, $p_{out}$ and $\textrm{rms}_{out}$ assess its quality as an upper bound.  \begin{table}[ht]
    \centering
    \begin{tabular}{llrrrrr}
        \hline
        instance & bound & $\mathrm{rms}_{all}$ & $p_{out}$ & $\mathrm{rms}_{out}$ \\
        \hline
        toy      & freq  & 8.00        &    17.54  &     0.90    \\
        toy      & mc    & 10.41       &    9.62   &     2.86    \\
        P1\cite{Kondili1993}       & freq  & 12.61       &    18.08  &     5.80    \\
        P1       & mc    & 17.25       &    10.13  &     1.81    \\
        P2\cite{Karimi1997}       & freq  & 7.40        &    48.19  &     2.24    \\
        P2       & mc    & 13.68       &    40.56  &     1.10    \\
        P4\cite{Maravelias2003}       & freq  & 10.09       &    13.77  &     4.10    \\
        P4       & mc    & 11.40       &    11.91  &     2.67    \\
        P6\cite{Ierapetritou1998}       & freq  & 16.35       &    29.40  &     4.48    \\
        P6       & mc    & 20.16       &    21.27  &     3.23    \\
        \hline
        all      & freq  & 10.89       &    25.40  &     3.50    \\
        all      & mc    & 14.58       &    18.70  &     2.34    \\
        \hline
    \end{tabular}
    \caption{Average performance metrics for probability estimates - all
        instances} \label{tab:metrics-inst}
\end{table}

Table~\ref{tab:metrics-inst} shows values for all three instances averaged over demand scenarios and units for all tested STN instances.
It can be seen that the frequency approach is generally a better estimator for $p^f_j$ than the Markov chain approach (smaller values of $\textrm{rms}_{all}$) but also has a larger rate of misclassification $p_{out}$.
While $\textrm{rms}_{all}$ can be large due to noise in the rolling horizon solutions and $p_{out}$ values of up to $48\%$ show that $\bar{p}^f_j$ is not a perfect upper bound, $\textrm{rms}_{out}$ is generally small with average values of $3.50$ and $2.34\%$ for the Frequency and Markov chain approach respectively.
This means that, when $\bar{p}^f_j$ underestimates $p^f_j$, it does not do so by much.
Considering the noise introduced by non-optimal solutions and the rolling horizon framework, the error introduced by estimating $p^f_j$ through $\bar{p}^f_j$ is small.
Because it is a slightly better upper bound, the Markov chain approach was used for all subsequent experiments unless mentioned otherwise.

\begin{figure}[htbp]
    \centering
    \includegraphics[width=\textwidth]{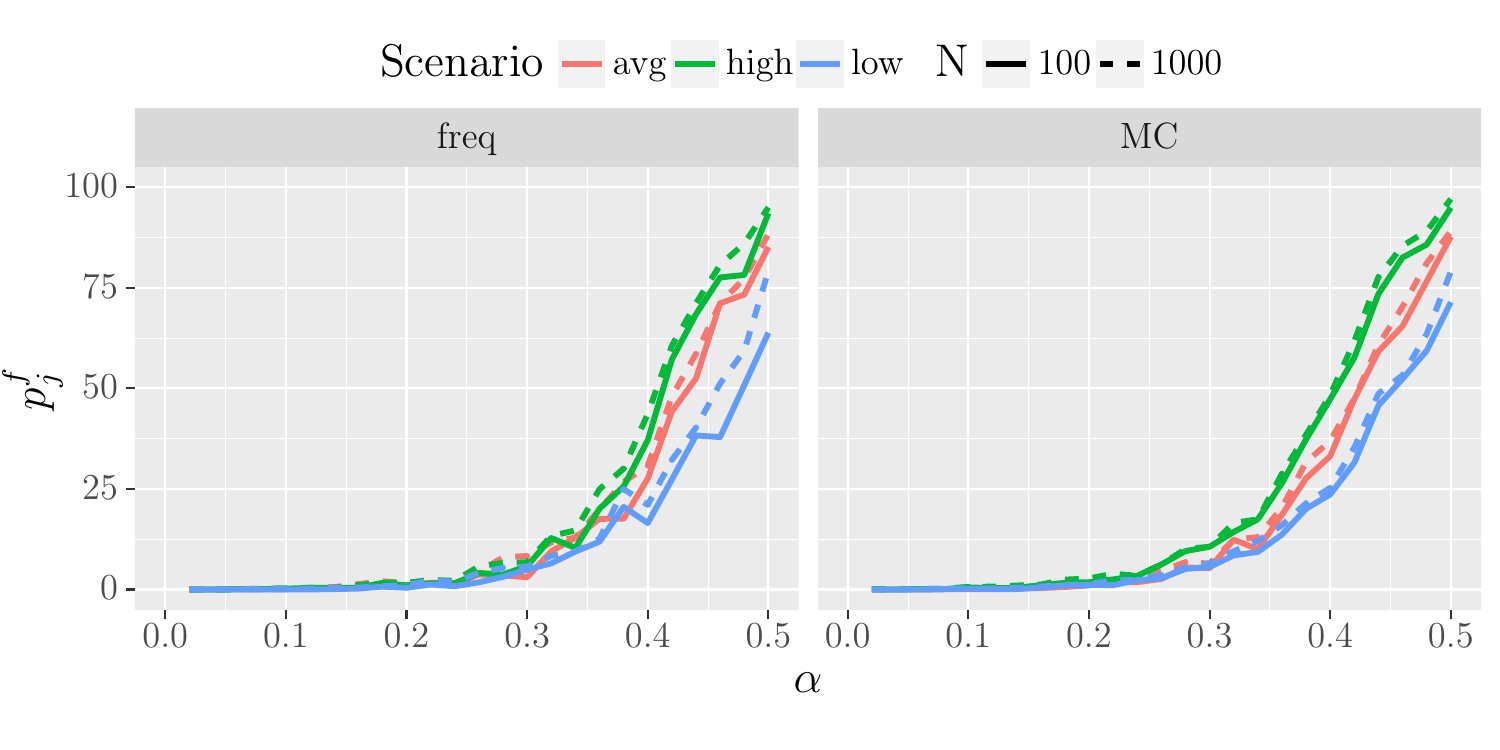}
    \caption{Effect of number of Monte-Carlo samples $N$ on failure probability
    $p^f_j$ of Reactor~1 (Instance~P1\cite{Kondili1993}).}
    \label{fig:100vs1000}
\end{figure}
Both probability estimation approaches are dependent on the number $N$ of operating mode sequences generated.
Fig.~\ref{fig:100vs1000} shows that increasing $N$ from $100$ to $1000$ for Reactor~1 in Instance~P1\cite{Kondili1993} only has a small effect on $p^f_j$, especially for the Markov chain based approach.
$N = 100$ is therefore deemed sufficient.

\begin{figure}[p]
    \centering
    \includegraphics{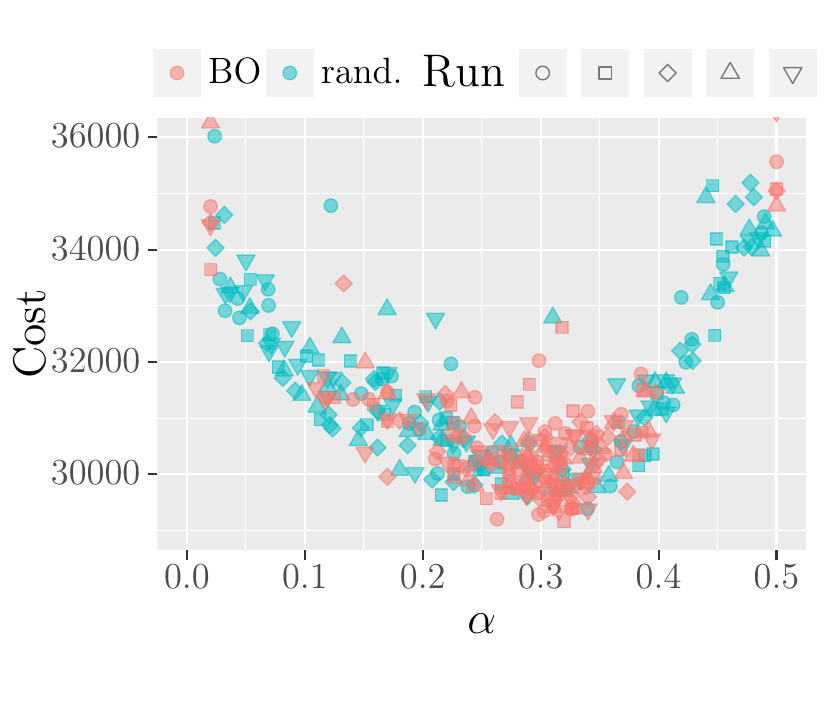}
    \caption{Bayesian optimization vs.~random search. Five runs of BO and
        random search were conducted. BO efficiently explores
        values near the optimal $\alpha$.}
    \label{fig:p-vs-alpha-bo-rand}
\end{figure}
\begin{figure}[p]
    \centering
    \includegraphics{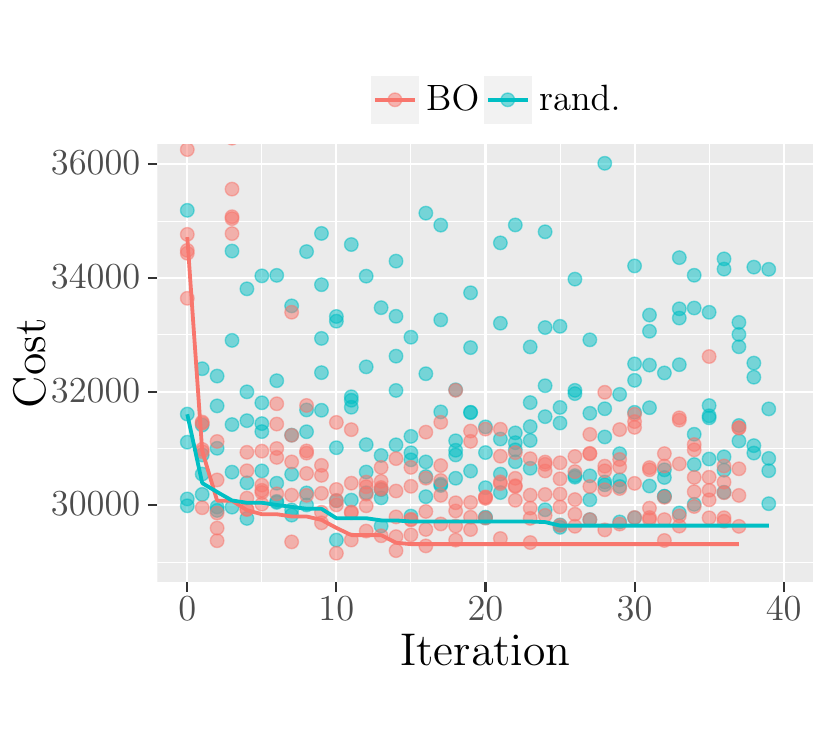}
    \caption{Bayesian optimization vs.~random search. Points are individual
        values obtained at a given iteration. Lines represent the best previously
    achieved solution value averaged over five runs. BO consistently finds
better solutions.}
    \label{fig:p-vs-iter-bo-rand}
\end{figure}
Figs.~\ref{fig:p-vs-alpha-bo-rand} and~\ref{fig:p-vs-iter-bo-rand} compare Bayesian optimization (BO) with random search for {Instance~P1\cite{Kondili1993}}.
Overall cost (Eqn.~\ref{eq:optalpha}) was evaluated over a horizon of 24 weeks and both Bayesian optimization and random search were repeated five times.
For the Bayesian optimization, four points were sampled evenly from the interval $\alpha \in [0.02, 0.5]$ initially.
Fig.~\ref{fig:p-vs-alpha-bo-rand} shows the obtained objective values as a function of $\alpha$.
There is clearly a trade-off between high cost of preventive maintenance in conservative solutions (small $\alpha$) and high cost of corrective maintenance for less conservative but also less robust solutions (large $\alpha$). Bayesian optimization very efficiently samples from the area around the optimal $\alpha \approx 0.3$ while random search naturally samples from the entire interval.
Fig~\ref{fig:p-vs-iter-bo-rand} shows the lowest objective value previously obtained as a function of the number of samples (averaged over five runs).
Bayesian optimization consistently finds lower cost solutions than random search and achieves a good compromise between preventive and corrective maintenance within about $20$ iterations.

\section{Conclusion}
\begin{figure}[htpb]
    \centering
    \includegraphics[width=\textwidth]{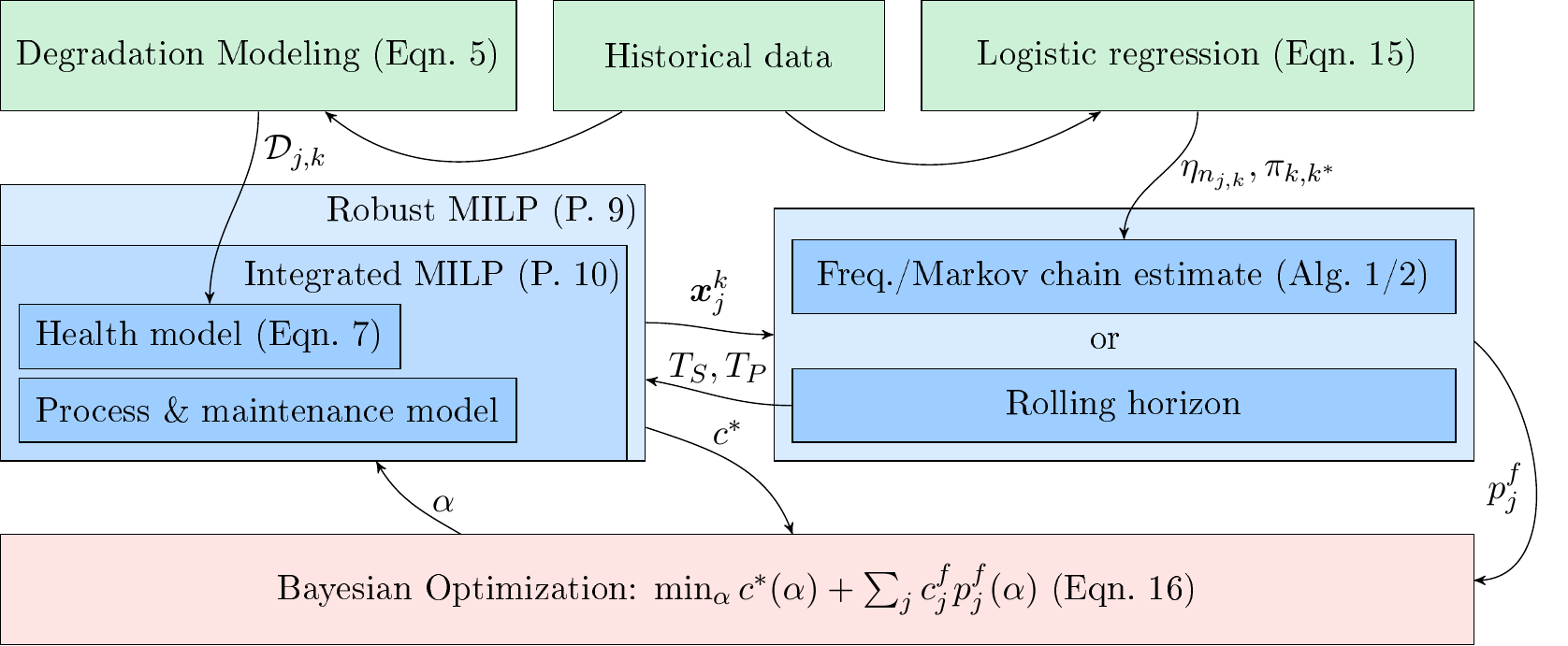}
    \caption{Flowchart of the proposed method.}
    \label{fig:flowchart-big}
\end{figure}
This work integrates equipment degradation effects in process level optimization problems.
The demonstrated methodology is summarized in Fig.~\ref{fig:flowchart-big}.
We combine commonly used methods from the Degradation Modeling literature, which allow unit health characteristics to be estimated and updated from data, with Robust Optimization.
This is highly relevant since, realistically, almost all equipment in chemical and manufacturing processes will be subject to performance degradation and failures.

Solving realistic integrated maintenance and production scheduling and planning problems is a hard task by itself, because models tend to be large and computationally expensive.
Combining such models with robust optimization increases the need for solving problems efficiently.
Furthermore, the scheduling task is highly repetitive and should become easier as historical data is collected.
To reduce computational expense, we show conditions where robust optimization problems can be solved by solving a deterministic approximation and develop data-based methods for estimating failure probabilities.

In the context of computationally expensive models, we show that Bayesian Optimization can be used effectively to optimize the uncertainty set in Robust optimization and balance the trade-off between preventive and corrective maintenance.

\section{Acknowledgement}
    This work was funded by the Engineering \& Physical Sciences Research Council (EPSRC) Center for Doctoral Training in High Performance Embedded and Distributed Systems (EP/L016796/1), an EPSRC/Schlumberger CASE studentship to J.W. (EP/R511961/1, voucher 17000145), and an EPSRC Research Fellowship to R.M. (EP/P016871/1).

\appendix
\section{Formulating the robust counterpart}
\label{sec:appa}
The reformulation of the semi-infinite constraints
\begin{subequations}
    \begin{align}
        & m_{j,t} s_{j}^{0} \leq s_{j,t}  & \forall t, j \in J, \tilde{d}_{j,k} \in \mathcal{D} \label{eq:cons1}\\
        & s_{j,t} \leq s_{j}^{max} + m_{j,t} \cdot (s_{j}^{0} - s_{j}^{max})
        & \forall t, j \in J, \tilde{d}_{j,k} \in \mathcal{U} \label{eq:cons2}\\
        & s_{j,t} \geq s_{j,t-\Delta t} + \sum_{k}{x_{j,k,t}\tilde{d}_{j,k}} + m_{j,t}
        \cdot (s_{j}^{0} - s_{j}^{max})  & \forall t, j \in J, \tilde{d}_{j,k}
        \in \mathcal{U} \label{eq:cons3}\\
        & s_{j,t} \leq s_{j,t-\Delta t} + \sum_{k}{x_{j,k,t}\tilde{d}_{j,k}}
        & \forall t, j \in J, \tilde{d}_{j,k} \in \mathcal{U} \label{eq:cons4}
    \end{align}
\end{subequations}
into a deterministic robust counterpart in this work is based on the approach by \citet{Lappas2016}.
$s_{j,t}$ is replaced by the affine decision rule
\begin{equation*}
    s_{j,t} = [s_{j,t}]_{0} + \sum_{k}{[s_{j,t}]_{k}\tilde{d}_{j,k}}
\end{equation*}
in all constraints, where $[s_{j,t}]_0$ and $[s_{j,t}]_k$ are coefficients which become new variables in the reformulated constraints.
The first constraint (Eqn.~\ref{eq:cons1}) can be reformulate in the following way:
\begin{equation*}
    \begin{aligned}
        & m_{j,t} s_{j}^{0} \leq [s_{j,t}]_{0} + \sum_{k}{[s_{j,t}]_{k}\tilde{d}_{j,k}}
        &                                                \\
        & \Rightarrow - \sum_{k}{[s_{j,t}]_{k}\tilde{d}_{j,k}} \leq [s_{j,t}]_{0} - m_{j,t} s_{j}^{0}    \\
        & \Rightarrow \Theta^* \leq [s_{j,t}]_{0} - m_{j,t} s_{j}^{0},
    \end{aligned}
\end{equation*}
where
\begin{equation*}
    \begin{aligned}
        & \Theta^{*} = \max_{\tilde{d}_{j,k}}
        && - \sum_{k}{[s_{j,t}]_{k}\tilde{d}_{j,k}}\\
        & \text{s.t}
        && -\tilde{d}_{j,k} \leq - \bar{d}_{j,k}(1-\epsilon)
        &   &   & \forall k \\
        &
        && \tilde{d}_{j,k} \leq \bar{d}_{j,k}(1+\epsilon)
        &   &   & \forall k.
    \end{aligned}
\end{equation*}
The dual of this is
\begin{equation*}
    \begin{aligned}
        & \Theta^{*} = \min_{u^{1,j,t}, l^{1,j,t}}
        && \sum_{k}{\bar{d}_{j,k}\left[(1+\epsilon)u^{1,j,t}_{k} -
        (1-\epsilon)l^{1,j,t}_{k}\right]}\\ & \text{s.t}
        && u^{1,j,t}_{k} - l^{1,j,t}_{k} \geq -[s_{j,t}]_{k}
        &   &   & \forall k,
    \end{aligned}
\end{equation*}
with dual variables $u^{1,j,t}_k$ and $l^{1,j,t}_k$.
Dropping the minimization leads to the final reformulation:
\begin{equation*}
    \begin{aligned}
        & \sum_{k}{\bar{d}_{j,k}\left[(1+\epsilon)u^{1,j,t}_{k} - (1-\epsilon)l^{1,j,t}_{k}\right]} \leq [s_{j,t}]_{0} - m_{j,t} s_{j}^{0}
        &
        & \forall t, j \in J \\
        &  u^{1,j,t}_{k} - l^{1,j,t}_{k} \geq -[s_{j,t}]_{k}
        &
        & \forall j,t,    k     \\
    \end{aligned}
\end{equation*}
Similar analysis for the second inequality (Eqn.~\ref{eq:cons2}) leads to reformulation
\begin{equation*}
    \begin{aligned}
        & \sum_{k}{\bar{d}_{j,k}\left[(1+\epsilon)u^{2,j,t}_{k} - (1-\epsilon)l^{2,j,t}_{k}\right]}\\
        & \leq s_{j}^{max} - [s_{j,t}]_{0} + m_{j,t}\cdot (s_{j}^{0} - s_{j}^{max})
        &
        & \forall t, j \in J \\
        &  u^{2,j,t}_{k} - l^{2,j,t}_{k} \geq [s_{j,t}]_{k}
        &
        & \forall j,t,k,
    \end{aligned}
\end{equation*}
the third inequality (Eqn.~\ref{eq:cons3}) yields
\begin{equation*}
    \begin{aligned}
        & \sum_{k}{\bar{d}_{j,k}\left[(1+\epsilon)u^{3,j,t}_{k} -
        (1-\epsilon)l^{3,j,t}_{k}\right]}\\
        & \leq [s_{j,t}]_{0} -
        [s_{j,t-\Delta t}]_{0} + m_{j,t} s_{j}^{max}
        &
        & \forall t, j \in J \\
        &  u^{3,j,t}_{k} - l^{3,j,t}_{k} \geq [s_{j,t-\Delta t}]_{k} - [s_{j,t}]_{k} + x_{j,k,t}
        &
        & \forall j,t,k,
    \end{aligned}
\end{equation*}
and the fourth (Eqn.~\ref{eq:cons4})
\begin{equation*}
    \begin{aligned}
        & \sum_{k}{\bar{d}_{j,k}\left[(1+\epsilon)u^{4,j,t}_{k} - (1-\epsilon)l^{4,j,t}_{k}\right]}\\
        & \leq - [s_{j,t}]_{0} + [s_{j,t-\Delta t}]_{0}
        &
        & \forall t, j \in J \\
        &  u^{4,j,t}_{k} - l^{4,j,t}_{k} \geq - [s_{j,t-\Delta t}]_{k} + [s_{j,t}]_{k} - x_{j,k,t}
        &
        & \forall j,t,k.
    \end{aligned}
\end{equation*}

\section{Equivalence to deterministic optimization with maximal parameters}
\label{sec:appb}

Note that for convenience and readability the index $j$ has been dropped in all
equations in this appendix.

Consider the case where the cost, process model, and maintenance model in
Problem~9 are not functions of the uncertain parameters $\tilde{d}_{k}$:
\begin{subequations}
    \label{eq:rob-app}
    \begin{align}
        &\min_{\bm{x}, \bm{m}}
        && \text{cost}(\bm{x}, \bm{m})
        &   &   &                                        \\
        & \text{s.t}
        && \text{process model}(\bm{x}, \bm{m})
        &   &   &                                        \\
        &
        && \text{maintenance model}(\bm{x}, \bm{m})
        &   &   &                                        \\
        &
        && m_t s_0 \leq [s_t]_0 + \sum{[s_t]_{k}\tilde{d}_k},
        &   &   & \forall t, \tilde{d}_k \in \mathcal{U} \label{eq:rob-c1}\\
        &
        && [s_t]_0 + \sum{[s_{t}]_k\tilde{d}_k} \leq s^{max} + m_t(s^0 - s^{max}),
        &   &   & \forall t, \tilde{d}_k \in \mathcal{U} \label{eq:rob-c2}\\
        &
        &
        \begin{split}
            [s_t]_0 + \sum{[s_{t}]_k\tilde{d}_k} \geq [s_{t-1}]_0 + \sum{[s_{t-1}]_k\tilde{d}_k} \\+ \sum_{k}x_{k,t}\tilde{d}_k \\ + m_t(s^0-s^{max}),
        \end{split}
        &   &   & \forall t, \tilde{d}_k \in \mathcal{U} \label{eq:rob-c3}\\
        &
        && [s_t]_0 + \sum{[s_{t}]_k\tilde{d}_k} [s_{t-1}]_0+ \sum{[s_{t-1}]_k\tilde{d}_k} + \sum_{k}x_{k,t-1}\tilde{d}_k,
        &   &   & \forall t, \tilde{d}_k \in \mathcal{U}\label{eq:rob-c4}
    \end{align}
\end{subequations}
Furthermore consider a deterministic version of Problem~\ref{eq:rob-app}
\begin{subequations}
    \label{eq:det-app}
    \begin{align}
        &\min_{\bm{x}, \bm{m}}
        && \text{cost}(\bm{x}, \bm{m})
        &   &   &           \\
        & \text{s.t}
        && \text{process model}(\bm{x}, \bm{m})
        &   &   &           \\
        &
        && \text{maintenance model}(\bm{x}, \bm{m})
        &   &   &           \\
        &
        && m_t s^0 \leq s_t,
        &   &   & \forall t \\
        &
        && s_t \leq s^{max} + m_t(s^0 - s^{max}),
        &   &   & \forall t \\
        &
        && s_t \geq s_{t-1} + \sum_{k}x_{k,t}d_k^{max} + m_t(s^0-s^{max}),
        &   &   & \forall t \\
        &
        && s_t \leq s_{t-1} + \sum_{k}x_{k,t}d_k^{max},
        &   &   & \forall t
    \end{align}
\end{subequations}
in which $\tilde{d}_k$ has been replaced by
\begin{equation*}
    d_k^{max} = \max_{\tilde{d}_k \in \mathcal{U}}\tilde{d}_k.
\end{equation*}

\begin{theorem}
    Given that cost, process model, and maintenance model are not functions of
    $\tilde{d}_{j,k}$ and that $s^0_j \leq s^{init}_j = s_{j,t=t_0} \leq
    s^{max}_j$ and $\tilde{d}_{j,k} \geq 0, \forall \tilde{d}_{j,k} \in
    \mathcal{U}$, then a feasible solution $(\bm{x} = [x_{k,t},\ldots],
    \bm{m}=[m_t,\ldots], \bm{h} = [s_t])$ to Problem~\ref{eq:det-app} forms a
    feasible solution $(\bm{x} = [x_{k,t},\ldots], \bm{m}=[m_t,\ldots], \bm{h}
    = [[s_t]_0, [s_t]_k])$ to Problem~\ref{eq:rob-app} with
\begin{subequations}
    \begin{align}
        [s_t]_0 &=
        \begin{cases}
            s^{init} & t < t_{m,0}    \\
            s^{0}    & t \geq t_{m,0}
        \end{cases} \label{eq:A0t}\\
        [s_{t}]_k &= \sum_{t'=t_{m,t}}^t{x_{k,t}}, \label{eq:Akt}
    \end{align}
\end{subequations}
where $s^{init} = s(t=0)$, $t_{m,0}$ is the first point in time at which maintenance is performed, and $t_{m,t}$ is the most recent point in time at which maintenance was performed.
\end{theorem}
\begin{proof}
    First we show that the Inequality~\ref{eq:rob-c1}
    \begin{equation}
        \tag{\ref{eq:rob-c1}}
        m_t s^0 \leq [s_t]_0 + \sum{[s_{t}]_k\tilde{d}_k}
    \end{equation}
    holds for any $\tilde{d}_k \geq 0$ given $(\bm{x} = [x_{k,t},\ldots], \bm{m}=[m_t,\ldots], \bm{h} = [[s_t]_0, [s_t]_k])$: From the assumption $s^0 \leq s^{init} \leq s^{max}$ and Eqn.~\ref{eq:A0t} it follows that $s^0 \leq [s_t]_0 \leq s^{max}$. Furthermore, it directly follows from Eqn.~\ref{eq:Akt} that $[s_{t}]_k \geq 0$. Eqn.~\ref{eq:rob-c1} is therefore guaranteed to hold for any $\tilde{d}_k \in \mathcal{U}, \tilde{d}_k \geq 0$.

    Next, we show that Inequalities~\ref{eq:rob-c3} and~\ref{eq:rob-c4},
    \begin{equation}
        \tag{\ref{eq:rob-c3}}
        [s_t]_0 + \sum{[s_{t}]_k\tilde{d}_k} \geq [s_{t-1}]_0 + \sum{[s_{t-1}]_k\tilde{d}_k} + \sum_{k}x_{k,t}\tilde{d}_k + m_t(s^0-s^{max})
    \end{equation}
    and
    \begin{equation}
        \tag{\ref{eq:rob-c4}}
        [s_t]_0 + \sum{[s_{t}]_k\tilde{d}_k} \leq [s_{t-1}]_0 + \sum{[s_{t-1}]_k\tilde{d}_k} + \sum_{k}x_{k,t}\tilde{d}_k
    \end{equation}
    respectively, hold for any $\tilde{d}_k \in \mathcal{U}$ as long as Inequality~\ref{eq:rob-c2} is satisfied.  We first assume that $m_t = 0$. In this case $[s_{t}]_k = [s_{t-1}]_k + x_{k,t}$ (from Eqn.~\ref{eq:Akt}), $[s_t]_0 = [s_{t-1}]_0$ (from Eqn.~\ref{eq:A0t}) and Eqns.~\ref{eq:rob-c3} and~\ref{eq:rob-c4} simplify to
    \begin{equation*}
        [s_t]_0 \geq [s_{t-1}]_0
    \end{equation*}
    and
    \begin{equation*}
        [s_t]_0 \leq [s_{t-1}]_0,
    \end{equation*}
    which is true for any $\tilde{d}_k$. Next we assume that $m_t = 1$. In this case $[s_t]_0 = s^0$ (from Eqn.~\ref{eq:A0t}), $[s_t]_k = 0$ (from Eqn.~\ref{eq:Akt}), and $x_{k,t} = 0$ (assuming that a unit can not be operated while maintenance is performed). Substituting this into Eqns.~\ref{eq:rob-c3} and~\ref{eq:rob-c4} and rearranging yields
    \begin{equation}
        \label{eq:ineq3-s}
        [s_{t-1}]_0 + \sum{[s_{t-1}]_k\tilde{d}_k} \leq s^{max}
    \end{equation}
    and
    \begin{equation}
        \label{eq:ineq4-s}
        s^0 \leq [s_{t-1}]_0 + \sum{[s_{t-1}]_k\tilde{d}_k}
    \end{equation}
    respectively. Eqn.~\ref{eq:ineq3-s} is guaranteed to be satisfied as long as Inequality~\ref{eq:rob-c2} holds for $t=t-1$ and Eqn.~\ref{eq:ineq4-s} holds for any $\tilde{d}_k \in \mathcal{U}, \tilde{d}_k \geq 0$ since $[s_{t-1}]_0 \geq S_0$ and $[s_{t-1}]_k \geq 0$.

    Finally we show that the Inequality~\ref{eq:rob-c2}
    \begin{equation}
        \tag{\ref{eq:rob-c2}}
        [s_t]_0 + \sum{[s_{t}]_k\tilde{d}_k} \leq s^{max} + m_t(s^0 - s^{max})
    \end{equation}
    holds for any $\tilde{d}_k \in \mathcal{U}$. To this end we notice that
    \begin{equation*}
        \argmax_{\tilde{d}_k \in \mathcal{U}}{[s_t]_0 + \sum_k{[s_{t}]_k\tilde{d}_k} = d_k^{max}}
    \end{equation*}
    since $[s_{t}]_k \geq 0$.
    Therefore, if Eqn.~\ref{eq:rob-c2} holds for $\tilde{d}_k = d_k^{max}$, it holds for any $\tilde{d}_k \in \mathcal{U}$.
    Since $(\bm{x} = [x_{k,t},\ldots], \bm{m}=[m_t,\ldots], \bm{h} = [s_t])$ is a solution to Problem~\ref{eq:det-app},
    \begin{equation}
        s_t \leq s^{max} + m_t(s^0 - s^{max}).
    \end{equation}
    Lastly, noticing that the definition of $[s_t]_0$ and $[s_t]_k$ (Eqns.~\ref{eq:A0t} and~\ref{eq:Akt}) ensure that $s_t$ can always be decomposed as
    \begin{equation*}
        s_t = [s_t]_0 + \sum_k{[s_{t}]_kd_k^{max}}
    \end{equation*}
    if $s_t$ satisfies Problem~\ref{eq:det-app}, it follows that
    \begin{equation*}
        s_t = [s_t]_0 + \sum_k{[s_{t}]_kd_k^{max}} \leq s^{max} + m_t(s^0 - s^h{max})
    \end{equation*}
    and Eqn.~\ref{eq:rob-c2} holds for all $\tilde{d}_k \in \mathcal{U}$.
\end{proof}

\section{Instances of the STN}
\label{sec:appc}
\begin{table}[H]
    {\tiny
        \begin{tabularx}{\textwidth}{r l l l l l l l l l} \hline \hline
            States        & Feed A   & Feed B   & Feed C   & Hot A & Int. BC & Int. AB & Impure E & Prod. 1  & Prod. 2  \\ \hline
            Capacity [kg] & $\infty$ & $\infty$ & $\infty$ & 100   & 200     & 150     & 100      & $\infty$ & $\infty$ \\
            Initial [kg]  & $\infty$ & $\infty$ & $\infty$ & 0     & 0       & 0       & 0        & 0        & 0        \\
            Storage cost  & 0        & 0        & 0        & 1     & 1       & 1       & 1        & 5        & 5        \\ \hline \hline
        \end{tabularx}
        \begin{tabularx}{\textwidth}{l l l l l}
            Units            & Heater & Reactor 1 & Reactor 2 & Still \\ \hline
            $v^{min}_j$ [kg] & 40     & 32        & 20        & 80    \\
            $v^{max}_j$ [kg] & 100    & 80        & 50        & 200   \\
            $s^{max}_j$      & 80     & 150       & 160       & 100   \\
            $s^{init}_j$     & 30     & 50        & 120       & 40    \\
            $\tau_j$ [hr]    & 15     & 21        & 24        & 15    \\
            $c^{maint}_j$    & 300    & 900       & 2000      & 1200  \\
            $c^f_j$          & 2000   & 3000      & 3000      & 1500  \\ \hline
        \end{tabularx}\\
        \begin{tabularx}{\textwidth}{l l l l l l l l l l}
            \hline
            Task & Mode & \multicolumn{8}{c}{Unit}\\
            \cline{3-10}
            & & \multicolumn{2}{l}{Heater} & \multicolumn{2}{l}{Reactor 1} & \multicolumn{2}{l}{Reactor 2} & \multicolumn{2}{l}{Still}\\
            \cline{3-10}
                    &        & $p_{i,j,k}$ & $\bar{d}_{i,j,k}/\sigma_{i,j,k}$ & $p$ & $\bar{d}/\sigma$ & $p$ & $\bar{d}/\sigma$ & $p$ & $\bar{d}/\sigma$ \\
            \hline
            Heating        & Slow   & 9           & 1/0.27                       &             &                              &             &                              &             &                              \\
                    & Normal & 6           & 2/0.54                       &             &                              &             &                              &             &                              \\
                        & Fast   & 3           & 3/0.81                       &             &                              &             &                              &             &                              \\
            Reaction 1    & Slow   &             &                              & 27          & 4/1.08                       & 30          & 4/1.08                       &             &                              \\
                    & Normal &             &                              & 15          & 5/1.35                       & 18          & 5/1.35                       &             &                              \\
                        & Fast   &             &                              & 8           & 8/2.43                       & 12          & 10/2.7                       &             &                              \\
            Reaction 2    & Slow   &             &                              & 36          & 1/0.27                       & 33          & 2/0.54                       &             &                              \\
                        & Normal &             &                              & 21          & 3/0.81                       & 18          & 4/1.08                       &             &                              \\
                        & Fast   &             &                              & 15          & 5/1.35                       & 12          & 4/1.08                       &             &                              \\
            Reaction 3    & Slow   &             &                              & 30          & 3/0.81                       & 24          & 2/0.54                       &             &                              \\
                        & Normal &             &                              & 18          & 7/1.89                       & 21          & 5/1.35                       &             &                              \\
                        & Fast   &             &                              & 6           & 8/2.16                       & 12          & 9/2.43                       &             &                              \\
            Separation    & Slow   &             &                              &             &                              &             &                              & 15          & 2/0.54                       \\
                        & Normal &             &                              &             &                              &             &                              & 9           & 5/1.35                       \\
                        & Fast   &             &                              &             &                              &             &                              & 6           & 6/1.62                       \\
            \hline
        \end{tabularx}
        \begin{tabularx}{\textwidth}{l l l l l l l}
            \hline
            Period & \multicolumn{6}{c}{Scenario}\\ \cline{2-7}
            & \multicolumn{2}{c}{Low} & \multicolumn{2}{c}{Average} & \multicolumn{2}{c}{High}\\ \cline{2-7}
                & Product 1 & Product 2 & Product 1 & Product 2 & Product 1 & Product 2 \\ \hline
            1    & 76        & 136       & 150       & 200       & 190       & 294       \\
            2    & 116       & 162       & 88        & 150       & 231       & 323       \\
            3    & 101       & 115       & 125       & 197       & 198       & 307       \\
            4    & 91        & 141       & 67        & 296       & 217       & 335       \\
            5    & 60        & 147       & 166       & 191       & 181       & 293       \\
            6    & 60        & 103       & 203       & 193       & 244       & 328       \\
            7    & 54        & 148       & 90        & 214       & 243       & 326       \\
            8    & 110       & 113       & 224       & 294       & 182       & 296       \\
            9    & 92        & 105       & 174       & 247       & 246       & 296       \\
            10    & 99        & 175       & 126       & 313       & 189       & 348       \\
            11    & 51        & 177       & 66        & 226       & 199       & 319       \\
            12    & 117       & 164       & 119       & 121       & 222       & 346       \\
            13    & 108       & 124       & 234       & 197       & 246       & 331       \\
            14    & 64        & 107       & 64        & 242       & 239       & 336       \\
            15    & 62        & 154       & 103       & 220       & 180       & 302       \\
            16    & 62        & 135       & 77        & 342       & 216       & 306       \\
            17    & 71        & 109       & 132       & 355       & 218       & 302       \\
            18    & 86        & 139       & 186       & 320       & 213       & 349       \\
            19    & 80        & 102       & 174       & 335       & 233       & 284       \\
            20    & 70        & 172       & 239       & 298       & 191       & 347       \\
            21    & 92        & 120       & 124       & 252       & 233       & 303       \\
            22    & 59        & 153       & 194       & 222       & 181       & 303       \\
            23    & 70        & 124       & 91        & 324       & 189       & 307       \\
            24    & 75        & 141       & 228       & 337       & 188       & 299       \\
            \hline \hline
        \end{tabularx}\\
        \caption{Instance~P1\cite{Kondili1993}}
        \label{tab:app-p1}
    }
    \end{table}
    \begin{table}[H]
    \centering
    \includegraphics{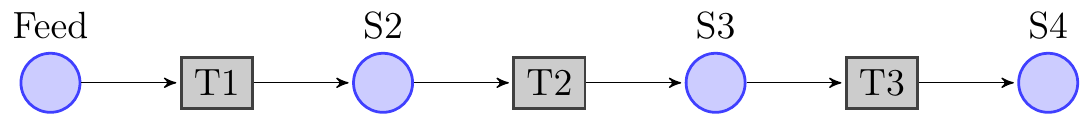}\\
    \vspace{1em}
    {\footnotesize
        \begin{tabularx}{\textwidth}{r l l l l} \hline \hline
            States        & S1       & S2       & S3       & S4       \\ \hline
            Capacity [kg] & $\infty$ & $\infty$ & $\infty$ & $\infty$ \\
            Initial [kg]  & $\infty$ & 0        & 0        & 0        \\
            Storage cost  & 0        & 1        & 1        & 1        \\ \hline
        \end{tabularx}
        \begin{tabularx}{\textwidth}{l l l l l l}
            \hline
            Units            & U1  & U2  & U3  & U4  & U5  \\ \hline
            $v^{min}_j$ [kg] & 0   & 0   & 0   & 0   & 0   \\
            $v^{max}_j$ [kg] & 100 & 150 & 200 & 150 & 150 \\
            $s^{max}_j$      & 120 & 100 & 150 & 90  & 80  \\
            $s^{init}_j$     & 10  & 70  & 45  & 60  & 30  \\
            $t_j$ [hr]       & 21  & 15  & 18  & 9   & 13  \\
            $c^{maint}_j$    & 600 & 600 & 500 & 400 & 400 \\ \hline
        \end{tabularx}\\
        \begin{tabularx}{\textwidth}{l l l l l l l l l l l l}
            \hline
            Task & Mode & \multicolumn{10}{c}{Unit}\\
            \cline{3-12}
            & & \multicolumn{2}{c}{U1} & \multicolumn{2}{c}{U2} & \multicolumn{2}{c}{U3} & \multicolumn{2}{c}{U4} & \multicolumn{2}{c}{U5}\\
            \cline{3-12}
               &        & $p_{i,j,k}$ & $\bar{d}_{i,j,k}/\sigma_{i,j,k}$ & $p$ & $\bar{d}/\sigma$ & $p$ & $\bar{d}/\sigma$ & $p$ & $\bar{d}/\sigma$ & $p$ & $\bar{d}/\sigma$ \\
            \hline
            T1 & Slow   & 33          & 3/0.66                       & 30  & 3/0.66       &     &              &     &              &     &              \\
            T1 & Normal & 25          & 5/1.35                       & 21  & 6/1.62       &     &              &     &              &     &              \\
            T1 & Fast   & 15          & 7/2.17                       & 12  & 9/2.79       &     &              &     &              &     &              \\
            T2 & Slow   &             &                              &     &              & 24  & 4/0.88       &     &              &     &              \\
            T2 & Normal &             &                              &     &              & 18  & 6/1.62       &     &              &     &              \\
            T2 & Fast   &             &                              &     &              & 12  & 10/3.1       &     &              &     &              \\
            T3 & Slow   &             &                              &     &              &     &              & 21  & 2/0.44       & 18  & 2/0.44       \\
            T3 & Normal &             &                              &     &              &     &              & 15  & 4/1.08       & 12  & 4/1.08       \\
            T3 & Fast   &             &                              &     &              &     &              & 9   & 7/2.17       & 6   & 6/1.86       \\
            \hline
        \end{tabularx}
        \begin{tabularx}{\textwidth}{l l l l l l l l}
            \hline
            Period & \multicolumn{3}{c}{Scenario} & Period & \multicolumn{3}{c}{Scenario}\\ \cline{2-4} \cline{6-8}
               & Low & Average & High &    & Low & Average & High \\ \hline
            1  & 725 & 1167    & 2002 & 13 & 446 & 947     & 1847 \\
            2  & 587 & 1110    & 2141 & 14 & 201 & 1426    & 2178 \\
            3  & 397 & 1087    & 1668 & 15 & 305 & 1090    & 2159 \\
            4  & 558 & 906     & 1977 & 16 & 447 & 1040    & 2015 \\
            5  & 411 & 1188    & 1692 & 17 & 378 & 917     & 1662 \\
            6  & 678 & 1191    & 1805 & 18 & 566 & 1190    & 1782 \\
            7  & 252 & 1436    & 2007 & 19 & 409 & 953     & 1646 \\
            8  & 415 & 1020    & 2174 & 20 & 797 & 1109    & 2135 \\
            9  & 539 & 1110    & 1713 & 21 & 605 & 1298    & 2113 \\
            10 & 414 & 1266    & 2162 & 22 & 413 & 1275    & 1760 \\
            11 & 214 & 1042    & 2155 & 23 & 550 & 1364    & 1958 \\
            12 & 612 & 1169    & 2018 & 24 & 362 & 1158    & 1779 \\
            \hline \hline
        \end{tabularx}\\}
        \caption{Instance~P2\cite{Karimi1997}}
    \end{table}
    \begin{table}[H]
        \centering
        \includegraphics{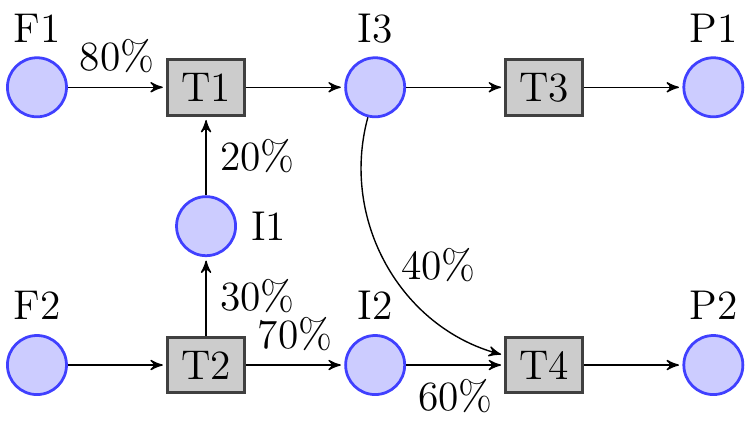}\\
        \vspace{1em}
        {\scriptsize
            \begin{tabularx}{\textwidth}{r l l l l l l l} \hline \hline
                States        & F1       & F2       & I1  & I2  & I3  & P1   & P2   \\ \hline
                Capacity [kg] & $\infty$ & $\infty$ & 200 & 100 & 500 & 1000 & 1000 \\
                Initial [kg]  & $\infty$ & $\infty$ & 0   & 0   & 0   & 0    & 0    \\
                Storage cost  & 0        & 0        & 1   & 1   & 1   & 5    & 5    \\ \hline
            \end{tabularx}
            \begin{tabularx}{\textwidth}{l l l l}
                \hline
                Units            & R1  & R2  & R3  \\ \hline
                $v^{min}_j$ [kg] & 40  & 25  & 40  \\
                $v^{max}_j$ [kg] & 80  & 50  & 80  \\
                $s^{max}_j$      & 70  & 120 & 70  \\
                $s^{init}_j$     & 10  & 10  & 10  \\
                $\tau_j$ [hr]    & 21  & 21  & 21  \\
                $c^{maint}_j$    & 600 & 600 & 600 \\ \hline
            \end{tabularx}\\
            \begin{tabularx}{\textwidth}{l l l l l l l l}
                \hline
                Task & Mode & \multicolumn{6}{c}{Unit}\\
                \cline{3-8}
                & & \multicolumn{2}{l}{R1} & \multicolumn{2}{l}{R2} & \multicolumn{2}{l}{R3}\\
                \cline{3-8}
                   &        & $p_{i,j,k}$ & $\bar{d}_{i,j,k}/\sigma_{i,j,k}$ & $p$ & $\bar{d}/\sigma$ & $p$ & $\bar{d}/\sigma$ \\
                \hline
                T1 & Slow   & 24          & 3/0.66                           & 24  & 3/0.66           &     &                  \\
                T1 & Normal & 15          & 5/1.35                           & 15  & 5/1.35           &     &                  \\
                T1 & Fast   & 9           & 8/2.16                           & 9   & 8/2.16           &     &                  \\
                T2 & Slow   & 36          & 3/0.66                           & 36  & 3/0.66           &     &                  \\
                T2 & Normal & 24          & 5/1.35                           & 24  & 5/1.35           &     &                  \\
                T2 & Fast   & 15          & 8/2.16                           & 15  & 8/2.16           &     &                  \\
                T3 & Slow   &             &                                  &     &                  & 12  & 3/0.66           \\
                T3 & Normal &             &                                  &     &                  & 9   & 5/1.35           \\
                T3 & Fast   &             &                                  &     &                  & 6   & 8/2.16           \\
                T4 & Slow   &             &                                  &     &                  & 24  & 3/0.66           \\
                T4 & Normal &             &                                  &     &                  & 15  & 5/1.35           \\
                T4 & Fast   &             &                                  &     &                  & 9   & 8/2.16           \\
                \hline
            \end{tabularx}
            \begin{tabularx}{\textwidth}{l l l l l l l l l l l l l l}
                \hline
                Period & \multicolumn{6}{c}{Scenario} & Period & \multicolumn{6}{c}{Scenario}\\ \cline{2-7} \cline{9-14}
                & \multicolumn{2}{c}{Low} & \multicolumn{2}{c}{Average} & \multicolumn{2}{c}{High} & & \multicolumn{2}{c}{Low} & \multicolumn{2}{c}{Average} & \multicolumn{2}{c}{High}\\ \cline{2-7} \cline{9-14}
                   & P1  & P2  & P1  & P2  & P1  & P2  &    & P1  & P2  & P1  & P2  & P1  & P2  \\ \hline
                1  & 190 & 102 & 271 & 231 & 311 & 305 & 13 & 197 & 192 & 213 & 280 & 314 & 376 \\
                2  & 172 & 156 & 230 & 282 & 381 & 347 & 14 & 120 & 100 & 257 & 270 & 370 & 392 \\
                3  & 102 & 103 & 274 & 226 & 381 & 321 & 15 & 180 & 115 & 286 & 223 & 335 & 370 \\
                4  & 130 & 172 & 289 & 281 & 310 & 310 & 16 & 178 & 186 & 201 & 281 & 386 & 358 \\
                5  & 130 & 104 & 270 & 212 & 317 & 371 & 17 & 135 & 138 & 288 & 289 & 398 & 309 \\
                6  & 174 & 192 & 205 & 205 & 339 & 328 & 18 & 140 & 115 & 284 & 253 & 304 & 396 \\
                7  & 167 & 194 & 260 & 248 & 379 & 348 & 19 & 128 & 104 & 200 & 273 & 334 & 351 \\
                8  & 185 & 175 & 259 & 282 & 317 & 392 & 20 & 155 & 179 & 275 & 210 & 326 & 300 \\
                9  & 179 & 180 & 211 & 233 & 300 & 387 & 21 & 158 & 120 & 298 & 253 & 315 & 397 \\
                10 & 131 & 120 & 261 & 292 & 346 & 326 & 22 & 160 & 186 & 252 & 284 & 301 & 396 \\
                11 & 104 & 196 & 271 & 212 & 364 & 364 & 23 & 125 & 105 & 223 & 220 & 396 & 378 \\
                12 & 104 & 188 & 202 & 289 & 309 & 346 & 24 & 187 & 134 & 205 & 285 & 316 & 393
                \\
                \hline \hline
            \end{tabularx}\\
            \caption{Instance~P4\cite{Maravelias2003}}
            }
        \end{table}

    \begin{table}[H]
        \centering
        \includegraphics{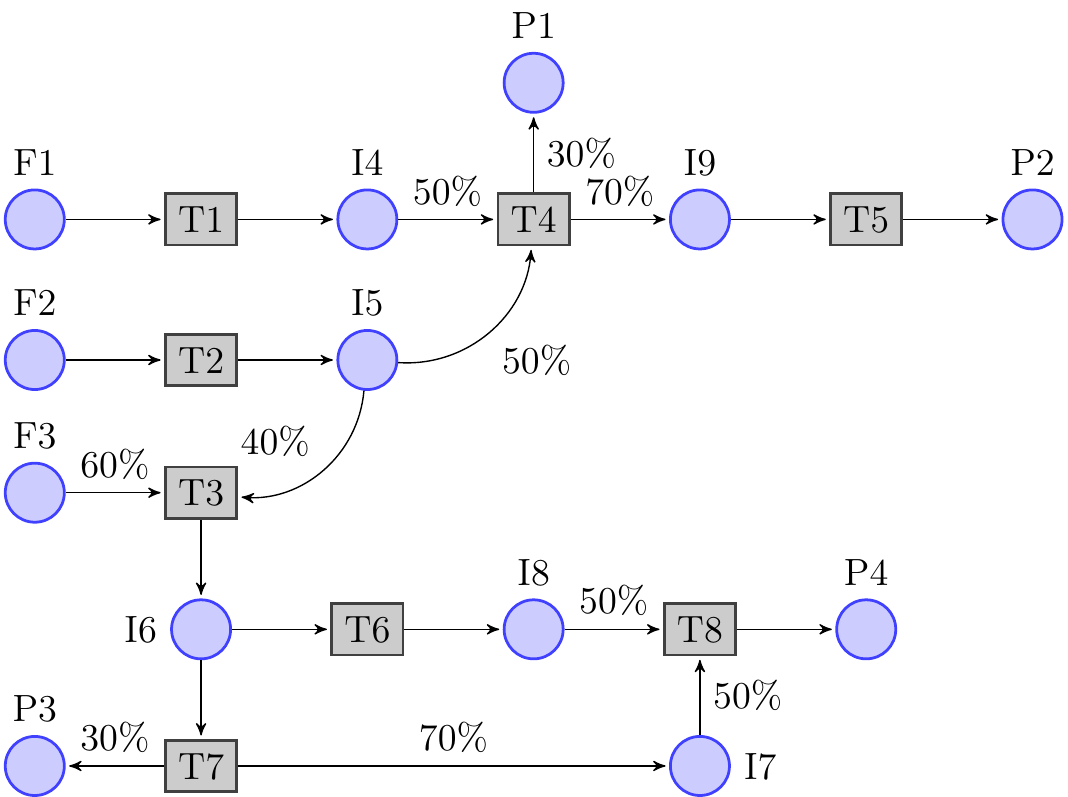}\\
        \vspace{1em}
        {\tiny
        \begin{tabularx}{\textwidth}{r l l l l l l l l l l l l l} \hline \hline
                States        & F1       & F2       & F3       & I4   & I5   & I6   & I7   & I8   & I9   & P1       & P2       & P3       & P4       \\ \hline
                Capacity [kg] & $\infty$ & $\infty$ & $\infty$ & 1000 & 1000 & 1500 & 2000 & 1000 & 3000 & $\infty$ & $\infty$ & $\infty$ & $\infty$ \\
                Initial [kg]  & $\infty$ & $\infty$ & $\infty$ & 0    & 0    & 0    & 0    & 0    & 0    & 0        & 0        & 0        & 0        \\
                Storage cost  & 0        & 0        & 0        & 1    & 1    & 1    & 1    & 1    & 1    & 5        & 5        & 5        & 5        \\ \hline
            \end{tabularx}
            \begin{tabularx}{\textwidth}{l l l l l l l}
                \hline
                Units            & R1   & R2   & R3   & R4   & R5   & R6   \\ \hline
                $v^{min}_j$ [kg] & 0    & 0    & 0    & 0    & 0    & 0    \\
                $v^{max}_j$ [kg] & 1000 & 2500 & 3500 & 1500 & 1000 & 4000 \\
                $s^{max}_j$      & 100  & 100  & 100  & 100  & 100  & 100  \\
                $s^{init}_j$     & 77   & 80   & 90   & 17   & 40   & 33   \\
                $\tau_j$ [hr]    & 21   & 21   & 21   & 21   & 21   & 21   \\
                $c^{maint}_j$    & 1000 & 1700 & 2000 & 1200 & 1000 & 2100 \\ \hline
            \end{tabularx}\\
            \begin{tabularx}{\textwidth}{l l l l l l l l l l l l l l}
                \hline
                Task & Mode & \multicolumn{12}{c}{Unit}\\
                \cline{3-14}
                & & \multicolumn{2}{l}{R1} & \multicolumn{2}{l}{R2} & \multicolumn{2}{l}{R3} & \multicolumn{2}{l}{R4} & \multicolumn{2}{l}{R5} & \multicolumn{2}{l}{R6}\\
                \cline{3-14}
        &        & $p_{i,j,k}$ & $\bar{d}_{i,j,k}/\sigma_{i,j,k}$ & $p$ & $\bar{d}/\sigma$ & $p$ & $\bar{d}/\sigma$ & $p$ & $\bar{d}/\sigma$ & $p$ & $\bar{d}/\sigma$ & $p$ & $\bar{d}/\sigma$ \\
                \hline
                T1 & Slow   & 18          & 3/0.66                       &     &              &     &              &     &              &     &              &     &              \\
                T1 & Normal & 12          & 5/1.35                       &     &              &     &              &     &              &     &              &     &              \\
                T2 & Slow   &             &                              &     &              &     &              & 24  & 3/0.66       &     &              &     &              \\
                T2 & Normal &             &                              &     &              &     &              & 15  & 5/1.35       &     &              &     &              \\
                T3 & Slow   &             &                              & 33  & 5/0.66       &     &              &     &              &     &              &     &              \\
                T3 & Normal &             &                              & 21  & 8/1.35       &     &              &     &              &     &              &     &              \\
                T4 & Slow   &             &                              &     &              & 42  & 5/0.66       &     &              &     &              &     &              \\
                T4 & Normal &             &                              &     &              & 21  & 11/1.35      &     &              &     &              &     &              \\
                T5 & Slow   &             &                              &     &              &     &              &     &              &     &              & 45  & 7/0.66       \\
                T5 & Normal &             &                              &     &              &     &              &     &              &     &              & 30  & 10/1.35      \\
                T6 & Slow   &             &                              &     &              &     &              &     &              & 18  & 3/0.66       &     &              \\
                T6 & Normal &             &                              &     &              &     &              &     &              & 12  & 5/1.35       &     &              \\
                T7 & Slow   &             &                              & 33  & 6/0.66       &     &              &     &              &     &              &     &              \\
                T7 & Normal &             &                              & 21  & 9/1.35       &     &              &     &              &     &              &     &              \\
                T8 & Slow   &             &                              &     &              &     &              &     &              &     &              & 45  & 6/0.66       \\
                T8 & Normal &             &                              &     &              &     &              &     &              &     &              & 30  & 10/1.35      \\
                \hline \hline
            \end{tabularx}
            \caption{Instance~P6\cite{Ierapetritou1998}}
        }
        \end{table}
        \begin{table}[H]
            \small
            \begin{tabularx}{\textwidth}{l l l l l l l l l l l l l}
                \hline
                Period & \multicolumn{12}{c}{Scenario}\\ \cline{2-13}
                & \multicolumn{4}{c}{Low} & \multicolumn{4}{c}{Average} & \multicolumn{4}{c}{High}\\ \cline{2-13}
                   & P1  & P2  & P3  & P4  & P1   & P2   & P3   & P4   & P1   & P2   & P3   & P4   \\ \hline
                1  & 715 & 501 & 801 & 933 & 1277 & 1356 & 1739 & 1349 & 1605 & 1844 & 1898 & 1631 \\
                2  & 593 & 878 & 888 & 739 & 1533 & 1361 & 1384 & 1374 & 1687 & 1882 & 1650 & 1732 \\
                3  & 743 & 995 & 817 & 563 & 1727 & 1400 & 1323 & 1351 & 1510 & 1805 & 1893 & 1699 \\
                4  & 620 & 963 & 636 & 698 & 1702 & 1701 & 1260 & 1605 & 1557 & 1717 & 1693 & 1908 \\
                5  & 991 & 612 & 535 & 686 & 1521 & 1424 & 1600 & 1315 & 1929 & 1769 & 1657 & 1918 \\
                6  & 919 & 819 & 914 & 626 & 1451 & 1412 & 1667 & 1406 & 1539 & 1818 & 1676 & 1521 \\
                7  & 648 & 799 & 920 & 549 & 1447 & 1471 & 1732 & 1620 & 1999 & 1663 & 1632 & 1603 \\
                8  & 609 & 728 & 969 & 925 & 1746 & 1444 & 1694 & 1512 & 1673 & 1987 & 1924 & 1742 \\
                9  & 741 & 575 & 604 & 814 & 1691 & 1456 & 1384 & 1305 & 1826 & 1559 & 1982 & 1882 \\
                10 & 968 & 682 & 853 & 532 & 1436 & 1297 & 1564 & 1356 & 1987 & 1737 & 1598 & 1910 \\
                11 & 624 & 789 & 816 & 728 & 1357 & 1730 & 1441 & 1638 & 1696 & 1660 & 1761 & 1778 \\
                12 & 840 & 929 & 700 & 733 & 1384 & 1283 & 1461 & 1423 & 1779 & 1722 & 1558 & 1655 \\
                13 & 516 & 790 & 705 & 743 & 1387 & 1594 & 1696 & 1533 & 1596 & 1528 & 1857 & 1745 \\
                14 & 556 & 643 & 974 & 890 & 1722 & 1493 & 1528 & 1533 & 1782 & 1829 & 1994 & 1512 \\
                15 & 940 & 715 & 797 & 638 & 1470 & 1377 & 1635 & 1303 & 1949 & 1738 & 1933 & 1782 \\
                16 & 894 & 896 & 693 & 853 & 1563 & 1467 & 1526 & 1565 & 1837 & 1593 & 1938 & 1852 \\
                17 & 792 & 994 & 509 & 647 & 1561 & 1593 & 1614 & 1531 & 1694 & 1869 & 1879 & 1593 \\
                18 & 725 & 718 & 856 & 789 & 1739 & 1681 & 1373 & 1255 & 1902 & 1663 & 1814 & 1953 \\
                19 & 820 & 886 & 971 & 531 & 1576 & 1706 & 1635 & 1628 & 1875 & 1988 & 1648 & 1512 \\
                20 & 950 & 788 & 580 & 859 & 1627 & 1358 & 1469 & 1694 & 1642 & 1519 & 1999 & 1595 \\
                21 & 641 & 773 & 681 & 877 & 1257 & 1433 & 1581 & 1420 & 1890 & 1942 & 1854 & 1826 \\
                22 & 793 & 963 & 950 & 634 & 1250 & 1641 & 1644 & 1404 & 1795 & 1628 & 1658 & 1961 \\
                23 & 504 & 741 & 531 & 671 & 1472 & 1617 & 1311 & 1519 & 1967 & 1768 & 1877 & 1739 \\
                24 & 830 & 529 & 819 & 594 & 1390 & 1645 & 1632 & 1614 & 1844 & 1945 & 1620 & 1563 \\
                \hline
            \end{tabularx}\\
            \caption{Instance~P6\cite{Ierapetritou1998} continued}
        \end{table}
        \begin{table}[H]
            \centering
            \includegraphics{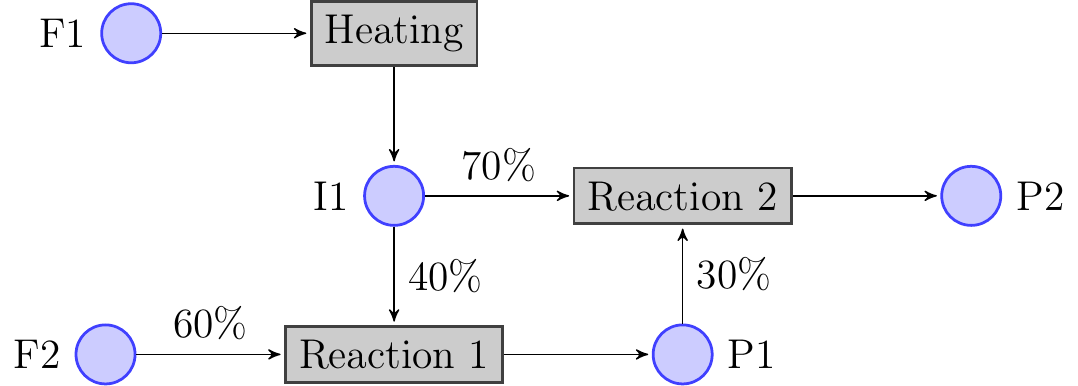}\\
            \vspace{1em}
            {\scriptsize
            \begin{tabularx}{\textwidth}{r l l l l l} \hline \hline
                States        & F1       & F2       & I1       & P1       & P2       \\ \hline
                Capacity [kg] & $\infty$ & $\infty$ & $\infty$ & $\infty$ & $\infty$ \\
                Initial [kg]  & $\infty$ & $\infty$ & 0        & 0        & 0        \\
                Storage cost  & 0        & 0        & 15       & 9        & 5        \\ \hline
            \end{tabularx}
            \begin{tabularx}{\textwidth}{l l l}
                \hline Units            & Heater & Reactor \\ \hline
                $v^{min}_j$ [kg] & 40     & 30      \\
                $v^{max}_j$ [kg] & 100    & 140     \\
                $s^{max}_j$      & 80     & 120     \\
                $s^{init}_j$     & 43     & 30      \\
                $\tau_j$ [hr]    & 2      & 3       \\
                $c^{maint}_j$    & 300    & 300     \\ \hline
            \end{tabularx}\\
            \begin{tabularx}{\textwidth}{l l l l l l}
                \hline
                Task & Mode & \multicolumn{4}{c}{Unit}\\
                \cline{3-6}
                & & \multicolumn{2}{l}{Heater} & \multicolumn{2}{l}{Reactor}\\
                \cline{3-6}
                           &        & $p_{i,j,k}$ & $\bar{d}_{i,j,k}/\sigma_{i,j,k}$ & $p$ & $\bar{d}/\sigma$ \\
                \hline
                Heating    & Slow   & 9           & 5.5/1.49                     &     &              \\
                Heating    & Normal & 6           & 11/2.97                      &     &              \\
                Reaction 1 & Slow   &             &                              & 6   & 7/1.89       \\
                Reaction 1 & Normal &             &                              & 4   & 9/2.43       \\
                Reaction 2 & Slow   &             &                              & 10  & 5/1.35       \\
                Reaction 2 & Normal &             &                              & 6   & 13/3.51      \\
                \hline
            \end{tabularx}\\
            \begin{tabularx}{\textwidth}{l l l l l l l l l l l l l l}
                \hline
                Period & \multicolumn{6}{c}{Scenario} & Period & \multicolumn{6}{c}{Scenario}\\ \cline{2-7} \cline{9-14}
                & \multicolumn{2}{c}{Low} & \multicolumn{2}{c}{Average} & \multicolumn{2}{c}{High} & & \multicolumn{2}{c}{Low} & \multicolumn{2}{c}{Average} & \multicolumn{2}{c}{High}\\ \cline{2-7} \cline{9-14}
                   & P1 & P2 & P1  & P2  & P1  & P2  &    & P1 & P2 & P1  & P2  & P1  & P2  \\ \hline
                1  & 52 & 61 & 121 & 113 & 184 & 201 & 13 & 49 & 45 & 141 & 128 & 180 & 225 \\
                2  & 60 & 32 & 137 & 149 & 228 & 207 & 14 & 29 & 44 & 115 & 120 & 210 & 192 \\
                3  & 42 & 59 & 149 & 128 & 204 & 218 & 15 & 69 & 35 & 148 & 106 & 220 & 192 \\
                4  & 51 & 40 & 146 & 127 & 192 & 200 & 16 & 53 & 51 & 125 & 148 & 224 & 225 \\
                5  & 24 & 42 & 105 & 103 & 206 & 225 & 17 & 38 & 49 & 102 & 136 & 229 & 218 \\
                6  & 34 & 52 & 111 & 100 & 194 & 212 & 18 & 47 & 77 & 112 & 125 & 204 & 201 \\
                7  & 32 & 55 & 105 & 143 & 223 & 216 & 19 & 54 & 53 & 141 & 134 & 197 & 226 \\
                8  & 30 & 62 & 141 & 140 & 182 & 218 & 20 & 26 & 77 & 118 & 131 & 223 & 215 \\
                9  & 65 & 30 & 144 & 128 & 185 & 197 & 21 & 61 & 68 & 116 & 128 & 199 & 183 \\
                10 & 32 & 69 & 147 & 122 & 218 & 180 & 22 & 31 & 71 & 132 & 101 & 212 & 196 \\
                11 & 59 & 37 & 119 & 119 & 182 & 197 & 23 & 53 & 42 & 107 & 121 & 181 & 222 \\
                12 & 24 & 44 & 133 & 139 & 218 & 195 & 24 & 39 & 47 & 113 & 143 & 220 & 211 \\
                \hline \hline
            \end{tabularx}\\
        \caption{Toy instance}
        \label{tab:app-toy}
        }
        \end{table}
        \clearpage
        Except for the toy instance, all instances are taken from the benchmark collection by \citet{Lappas2016}.
        All parameters were converted to a discrete time formulation and degradation parameters were added.
        For Instance~P1\cite{Kondili1993} parameters from \citet{Biondi2017} were used.
        The time steps and horizons used in each instance are given in
        Table~\ref{tab:instances-times}.
        \begin{table}[htb]
            \centering
            \begin{tabular}{l r r r r r}
                \hline
                Instance            & Toy  & P1\cite{Kondili1993}   & P2\cite{Karimi1997}   & P4\cite{Maravelias2003}   & P6\cite{Ierapetritou1998}   \\ \hline
                $T_S$               &  30  &  168  &  168  &  168  &  168  \\
                $\Delta t_S$        &   1  &    3  &    3  &    3  &    3  \\
                $T_P$               & 720  & 4032  & 4032  & 4032  & 4032  \\
                $\Delta t_P$        &  30  &  168  &  168  &  168  &  168  \\
                \hline
            \end{tabular}
            \caption{Time horizons of STN instances}
            \label{tab:instances-times}
        \end{table}

        The degradation of all units is assumed to follow a Wiener process. The distribution of increments is
        \begin{equation*}
        S_{j,t+p_{i,j,k}} - S_{j,t} = D_{i,j,k}, \quad D_{i,j,k} \sim \mathcal{N}\left(\bar{d}_{i,j,k},\sigma^2_{i,j,k}\right),
        \end{equation*}
    where $\bar{d}_{i,j,k}$ is the nominal amount of degradation when task $i$ is performed on unit $j$ in mode $k$.
        When no task is being processed the degradation signal is assumed to vary with zero mean and a small variance:
    \begin{equation}
        S_{j,t+\Delta t} - S_{j,t} = D_{0,\Delta t}, \quad D_{0, \Delta t} \sim \mathcal{N}\left(0, 0.05^2\Delta t\right).
    \end{equation}

    A detailed list of all parameter values used in each case study can be
    found in Tables~\ref{tab:app-p1} through~\ref{tab:app-toy}.

\section{Crossing probabilities of a Brownian motion for a piecewise linear boundary}
\label{sec:apppoetzel}

When the Wiener process is used as a degradation model, the failure probability $p^f_j$ can be calculated efficiently based on analytical result\cite{Poetzelberger1997, Bian2011, Bian2013}.
The probability of a Wiener process with piecewise constant parameters $\boldsymbol{\theta_{j,k}} = \left[\mu_{j,k}, \sigma_{j,k}\right]$ crossing a fixed threshold $s^{max}_j$ is equivalent to the probability of a standard Brownian motion $W(t)$ (a Wiener process with $\mathcal{N}\left(0,1\right)$ distributed increments) crossing a piecewise linear boundary.
The probability of $W(t)$ crossing a linear boundary $at+b$ is known to be inverse gaussian distributed
\begin{equation}
    P(W(t) \geq at+b, t \leq T) = 1 - \Phi(\frac{aT+b}{\sqrt{T}})
    + \exp^{-2ab}\Phi(\frac{aT-b}{\sqrt{T}}),
\end{equation}
where $\Phi(\cdot)$ is the standard normal distribution function \cite{Siegmund1986}.

Based on this, the probability of failure $p^{u,v}_j$ between two consecutive maintenance times $t_{m,u}$ and $t_{m,v}$ can be calculated as
\begin{equation*}
    \begin{aligned}
        p^{u,v}_j & = 1 - \mathbbm{E}h(\boldsymbol{y})\\ & = 1 -
        \prod_{l=1}^n\mathbbm{1}(y_l >0)
        \left(1-\exp\left[-\frac{2y_{l-1}y_l}{t_l-t_{l-1}}\right]\right),
    \end{aligned}
\end{equation*}
where $t_l, l \in \{1,\ldots,n\}$ are the $n$ points in time between $t_{m,u}$ and $t_{m,v}$ at which the operating mode $k$ changes\cite{Poetzelberger1997}.

Here $\boldsymbol{y}$ is a vector representing the values of $W(t)$ at each $t_l$.
It is defined as
\begin{equation*}
    \bm{y = c + MD^{1/2}u},
\end{equation*}
where $\boldsymbol{M}$ is a lower triangular matrix of ones,
\begin{equation*}
    \boldsymbol{D^{1/2}} = \text{diag}(\sqrt{t_1-t_{m,u}}, \sqrt{t_2-t_1},\ldots,\sqrt{t_{m,v}-t_{n}}),
\end{equation*}
$\boldsymbol{u}$ is a random vector with $u_l \sim
\mathcal{N}(0,\sigma^2_{j,k}(t_l))$, and $\boldsymbol{c}$ is the piecwise linear boundary
\begin{equation*}
    \boldsymbol{c} =
    (s_{j}^{max} - s_{j}^{init}) \cdot \left[1,\ldots,1\right]^{\top}
    - \boldsymbol{M}\text{diag}\left(0,\mu_{j,k}\left(t_1\right),\ldots,\mu_{j,k}\left(t_n\right)\right) \boldsymbol{\Delta t}
\end{equation*}
with
\begin{equation*}
    \boldsymbol{\Delta t} = \left[0, t_1-t_0,\ldots,t_n-t_{n-1}\right]^{\top}.
\end{equation*}

The overall probability of failure over the evaluation horizon $T$ can then be
calculated as
\begin{equation}
    \label{eq:pfailpb}
    p^f_j = 1 - \prod_{u = 1}^{n+1}\left(1-p_j^{u-1,l}\right),
\end{equation}
where $t_{m,0} = 0$, $t_{m,n+1} = T$, and $t_{m,u}, u \in \{1,\ldots,n\}$ are the $n$ points in time at which maintenance is carried out on unit $j$ in the evaluation horizon.
\section{Nomenclature}

\begin{longtable}{ l p{0.8\textwidth} }
    {$\boldsymbol{h}$}                      & {health variables}        \\
    {$\boldsymbol{m}$}                      & {maintenance variables}   \\
    {$\boldsymbol{x}$}                      & {process variables}       \\[8pt]
    \multicolumn{2}{l}{\textbf{\large Indices}}\\
    {$i$}                                   & {task}                    \\
    {$j$}                                   & {unit}                    \\
    {$k$}                                   & {operating mode}          \\
    {$s$}                                   & {state}                   \\
    {$t$}                                   & {time}                    \\[8pt]
    \multicolumn{2}{l}{\textbf{\large Sets}}\\
    {$I$}                                   & {set of tasks}            \\
    {$I_j$}                                 & {set of tasks $i$ available on unit $j$}  \\
    {$I_s$}                                 & {set of tasks $i$ consuming state $s$}    \\
    {$\bar{I}_s$}                           & {set of tasks $i$ producing state $s$}    \\
    {$J$}                                   & {set of process units}                     \\
    {$J_i$}                                 & {set of units $j$ on which task $i$ can be performed} \\
    {$K_i$}                                 & {set of operating modes allowed for task $i$}         \\
    {$K_j$}                                 & {set of operating modes $k$ available on unit $j$}    \\
    {$T$}                                   & {evaluation horizon}      \\
    {$T_P$}                                 & {planning horizon}        \\
    {$T_S$}                                 & {scheduling horizon}      \\
    {$\mathcal{U}$}                         & {uncertainty set}         \\
    {$\mathcal{X}$}                         & {set of operating mode sequences $\boldsymbol{x}^k_j$} \\[8pt]
    \multicolumn{2}{l}{\textbf{\large Discrete Variables}}\\
    {$m_{j,t}$} & {$1$ if maintenance is performed on unit $j$ at time $t$}\\
    {$n_{i,j,k,t}$} & {number of times task $i$ is performed on unit $j$ in mode $k$ in time period $t$}\\
    {$w_{i,j,k,t}$} & {$1$ if task $i$ starts on unit $j$ in mode $k$ at time $t$, $0$ otherwise}\\
    {$x_{j,k,t}$} & {$1$ if unit $j$ is operated in mode $k$ at time $t$}\\
    {$\boldsymbol{x}^k_j$} & {sequence of operating modes $[k_1, k_2, \ldots, k_T]$}\\
    {$\omega_{j,k,t}$} & {$1$ if unit $j$ operates in mode $k$ in period $t$, $0$ otherwise}\\[8pt]
    \multicolumn{2}{l}{\textbf{\large Continuous Variables}}\\
    {$a_{i,j,k,t}$} & {amount of material processed by task $i$ in unit $j$ in mode $k$ in time period $t$}\\
    {$b_{i,j,k,t}$} & {amount of material committed to task $i$ on unit $j$ in mode $k$ at time $t$}\\
    {$c^*$} & {minimal cost determined by solving Problem~\ref{eq:rob}}\\
    {$c^f_j$} & {cost of unit $j$ failing}\\
    {$D$} & {random variable modeling increment of $S(t)$}\\
    {$N_{j,k}$} & {random variable modeling $n_{j,k}$}\\
    {$n_{j,k}$} & {number of times mode $k$ occurs on unit $j$ in a given $\Delta t$}\\
    {$p^f_j$} & {probability of failure}\\
    {$\bar{p}^f_j$} & {estimated upper bound on $p^f_j$}\\
    {$q^{fin}_{s}$} & {quantity of state $s$ stored at end of planning horizon}\\
    {$q_{s,t}$} & {quantity of state $s$ stored at time $t$}\\
    {$S(t)$} & {stochastic process modeling $s^{meas}(t)$}\\
    {$s^n_{j,t}$} & {realization of $S_j(t)$ at time t}\\
    {$s^{fin}_j$} & {value of degradation signal at end of planning horizon}\\
    {$s^{meas}(t)$} & {measured degradation signal}\\
    {$S_t$} & {random variable modeling $s^{meas}(t)$ at time $t$}\\
    {$X^k_j(t)$} & {memoryless Markov chain modeling $\boldsymbol{x}^k_j$}\\
    {$X^k_{j,t}$} & {state of Markov chain at time $t$}\\
    {$\phi^d_s$} & {slack variable for unfulfilled demand of state $s$}\\
    {$\phi^q_{s,t}$} & {slack variable for storage capacity violation of state $s$ at time $t$}\\
    {$\boldsymbol{\psi}$} & {process/environmental parameters} \\[8pt]
    \multicolumn{2}{l}{\textbf{\large Parameters}}\\
    {$c^{maint}_j$} & {cost of maintenance for unit $j$}\\
    {$c^{storage}_s$} & {per unit cost of storage for state $s$}\\
    {$c_s$} & {storage capacity of state $s$}\\
    {$\mathcal{D}$} & {distribution of $D$}\\
    {$\bar{d}_{j,k}$} & {nominal value of $\tilde{d}_{j,k}$}\\
    {$\tilde{d}_{j,k}$} & {uncertain parameter modeling increment of $S(t)$}\\
    {$d^{max}_{j,k}$} & {maximum of $\tilde{d}_{j,k}$ in $\mathcal{U}$}\\
    {$N$} & {number of samples in Monte-Carlo simulation}\\
    {$p_{i,j,k}$} & {processing time of task $i$ on unit $j$ in mode $k$}\\
    {$r_{j,t}$} & {residual lifetime of unit $j$ at time $t$}\\
    {$[s_{j,t}]_0, [s_{j,t}]_k$} & {parameters for affine decision rule}\\
    {$s^0$} & {reset value degradation signal}\\
    {$s^{init}$} & {initial value of degradation signal}\\
    {$s^{max}$} & {failure threshold degradation signal}\\
    {$\bar{t}_P$} & {first time period in planning horizon}\\
    {$\bar{t}_S$} & {last time period in scheduling horizon}\\
    {$\Delta t_P$} & {length of planning period}\\
    {$\Delta t_S$} & {length of scheduling period}\\
    {$U$} & {large number}\\
    {$v^{max}_{i,j}$} & {maximum batch size for task $i$ on unit $j$}\\
    {$v^{min}_{i,j}$} & {minimum batch size for task $i$ on unit $j$}\\
    {$\alpha$} & {size parameter of $\mathcal{U}$}\\
    {$\delta_{s,t}$} & {demand for $s$ at time $t$}\\
    {$\epsilon_{j,k}$} & {size parameter of $\mathcal{U}$}\\
    {$\eta_{j,k}$} & {probability of $k$ occuring $n_{j,k}$ times}\\
    {$\boldsymbol{\theta}$} & {parameter vector of $\mathcal{D}$}\\
    {$\mu_{j,k}$} & {mean of $\mathcal{D}_{j,k}$}\\
    {$\pi_{k,k*}$} & {transition probability Markov chain}\\
    {$\bar{\rho}_{i,s}$} & {fraction of state $s$ of material produced by task $i$}\\
    {$\rho_{i,s}$} & {fraction of state $s$ of material consumed by task $i$}\\
    {$\sigma_{j,k}$} & {standard deviation of $\mathcal{D}_{j,k}$}\\
    {$\tau_j$} & {duration of maintenance on unit $j$}\\
\end{longtable}

\bibliographystyle{compactnat}
\bibliography{lit_manual,lit}

\end{document}